\documentclass[11pt]{article}
\usepackage[margin=4cm,footskip=30pt,headheight=30pt]{geometry}
\usepackage{amsmath,amsthm}
\usepackage{amssymb}
\usepackage{amscd,xspace}
\usepackage{comment}
\usepackage{nicefrac,todonotes}
\usepackage{color}
  \usepackage{hyperref}

\newtheorem{theorem}{Theorem}[section]
\newtheorem{lemma}[theorem]{Lemma}
\newtheorem{proposition}[theorem]{Proposition}
\newtheorem{corollary}[theorem]{Corollary}

\newtheorem{conjecture}[theorem]{Conjecture}
\newtheorem{itheorem}{Theorem}

{
\theoremstyle{definition}
\newtheorem{definition}[theorem]{Definition}
\newtheorem{example}[theorem]{Example}

\newtheorem{remark}[theorem]{Remark}

}

  \newcommand{\nc}{\newcommand}
  \newcommand{\renc}{\renewcommand}

\newcommand{\arxiv}[1]{\href{http://arxiv.org/abs/#1}{\tt\nolinkurl{arXiv:#1}}}
\nc{\ep}{\epsilon}
\nc{\hh}{h}
\nc{\fg}{\mathfrak g}
\nc{\fk}{\mathfrak k}
\nc{\fh}{\mathfrak h}
\nc{\kb}{\mathbb{C}}
\nc{\C}{\mathbb{C}}
\nc{\cO}{\mathcal{O}}
\newcommand{\Mirkovic}{Mirkovi\'c\xspace}
\nc{\Sym}{\operatorname{Sym}}
\nc{\om}{\omega}
\nc{\Z}{\mathbb{Z}}
\nc{\si}{\sigma}
\nc{\la}{\lambda}
\nc{\al}{\alpha}
\nc{\Gr}{\mathsf{Gr}}
\nc{\Gp}{G[t]}
\nc{\Gm}{{G_1[[t^{-1}]]}}
\nc{\gp}{\mathfrak{g}[t]}
\nc{\gm}{t^{-1}\mathfrak{g}[[t^{-1}]]}
\nc{\Grlmbar}{\Gr^{\overline{\lambda}}_\mu}
\nc{\Grblmbar}{\Gr^{\bla}_\mu}
\renc{\O}{\mathcal{O}}
\nc{\val}{\operatorname{val}}
\nc{\pp}{\mathbb{P}^1}
\nc{\excise}[1]{}
\nc{\bla}{{\underline{\boldsymbol{\la}}}}
\nc{\bz}{\mathbf{z}}
\nc{\vlam}{\vec{\lambda}}

\nc{\jcom}[1]{{\color{blue} Joel: #1}}
\nc{\bcom}[1]{{\color{red} Ben: #1}}
\nc{\ocom}[1]{{\color{red!50} Oded: #1}}
\nc{\acom}[1]{{\color{green} Alex: #1}}
\nc{\bysame}{\leavevmode\hbox to3em{\hrulefill}\thinspace}

\renc{\al}{\alpha}

\begin{document}

%\pagestyle{fancy}
%\chead[Yangians and quantizations of slices in the affine Grassmannian]{Kamnitzer, Webster,
%  Weekes and Yacobi}
%\cfoot[\thepage]{\thepage}
\begin{center} {\Large \bf Yangians and quantizations of slices\\ in the
    affine Grassmannian }
\end{center}
\bigskip

\noindent{\bf Joel Kamnitzer}\\
Dept.\ of Mathematics, University of Toronto;
{\tt jkamnitz@math.toronto.edu}\smallskip \\
{\bf Ben Webster}\footnote{Supported by NSF grant DMS-1151473 and  by NSA grant H98230-10-1-0199.}\\
Dept.\ of Mathematics, Northeastern University;
{\tt b.webster@neu.edu}\smallskip\\
{\bf Alex Weekes}\\
Dept.\ of Mathematics, University of Toronto; {\tt alex.weekes@utoronto.ca}\smallskip \\
{\bf Oded Yacobi}\\
Dept.\ of Mathematics, University of Toronto; {\tt oyacobi@math.toronto.edu}\smallskip \\
\bigskip\\

\renewcommand{\thefootnote}{\fnsymbol{footnote}}

{\small
\begin{quote}
\noindent {\em Abstract.}
We study quantizations of transverse slices to Schubert varieties in the affine Grassmannian.  The quantization is constructed using quantum groups called shifted Yangians --- these are subalgebras of the Yangian we introduce which generalize the Brundan-Kleshchev shifted Yangian to arbitrary type.  Building on ideas of Gerasimov-Kharchev-Lebedev-Oblezin, we prove that a quotient of the shifted Yangian quantizes a scheme supported on the transverse slices, and we formulate a conjectural description of the defining ideal of these slices which implies that the scheme is reduced.  This conjecture also implies the conjectural quantization of the Zastava spaces for $PGL_n$ of Finkelberg-Rybnykov.
\end{quote}
}

\section{Introduction}
\label{sec:introduction}

\renc{\theitheorem}{\Alph{itheorem}}

We initiate a program which relates the geometry of affine Grassmannians
with the representation theory of shifted Yangians.  More precisely, we study
slices in affine Grassmannians which arise naturally
in geometric representation theory; they correspond to weight spaces
of irreducible representations under the geometric Satake
correspondence.  Our main result is that certain subquotients of
Yangians quantize these slices.

There is a general program to study symplectic resolutions by means
of the representation theory of their quantizations, generalizing the interplay between between semisimple Lie algebras and nilpotent cones.   We believe that the representation theory of shifted Yangians and its relationship to the geometry of slices in the affine Grassmannian will prove to be a very fruitful area of inquiry.

\subsection{Slices in the affine Grassmannian}
Let $ G $ be a complex semisimple group and consider its \textbf{thick
  affine
Grassmannian} $ \Gr = G((t^{-1}))/G[t] $.  Attached to each pair of
dominant coweights $\la\geq \mu$, we have Schubert varieties
$\Gr^\la,\Gr^\mu\subset \Gr$, with $\Gr^\mu\subset
\overline{\Gr^\la}$.  The neighborhood in $\overline{\Gr^\la}$ of a
point in $\Gr^\mu$ is encapsulated in a transversal slice to the
latter variety in the former, which we denote by $\Grlmbar$.  This
slice is an important object of study in geometric representation
theory because under the geometric Satake correspondence it is related
to the $\mu$ weight space in the irreducible representation of $
G^\vee $ of highest weight $ \lambda $.

The Manin triple $ (\fg[t], t^{-1}\fg[[t^{-1}]], \fg((t^{-1})) ) $ provides $ \Gr $ with the structure of a Poisson variety. The slice $ \Grlmbar $ is an affine Poisson subvariety and thus, its coordinate ring is naturally a
Poisson algebra.  The purpose of this paper is to explicitly describe quantizations of this Poisson algebra.

\subsection{Quotients of shifted Yangians}
The slice $ \Grlmbar $ is defined as the intersection $
\overline{\Gr^\la} \cap \Gr_\mu $, where $ \Gr_\mu $ is an orbit of
the group $ \Gm$, the first congruence subgroup of $ G[[t^{-1}]] $.
Thus on the level of functions $ \O(\Grlmbar) $ is a quotient of $
\O(\Gr^\mu) $, and $ \O(\Gr^\mu) $ is a subalgebra of $ \O(\Gm) $.  In
order to quantize $ \Grlmbar $ we follow a three step procedure which
mirrors this construction.

We first construct a version  $Y $ of the Yangian, which is a subalgebra of the Drinfeld Yangian.  Next, we define natural
subalgebras $Y_\mu\subset Y$, called {\bf shifted Yangians}, quantize
$ \Gr_\mu $.  This generalizes the shifted Yangian for
$\mathfrak{gl}_n$ introduced by Brundan-Kleshchev
\cite{BK}. Finally, we define a quotient $Y_\mu^\lambda$ of $Y_\mu $
 using some remarkable representations of $ Y $ as difference
operators, constructed by Gerasimov-Kharchev-Lebedev-Oblezin
\cite{GKLO}.

\begin{itheorem} The algebras defined above are all quantizations of
  the analogous geometric objects.  That is:
  \begin{enumerate}
  \item The Yangian $Y$ quantizes $\Gm$.
  \item The shifted Yangian $Y_\mu$ quantizes $\Gr_\mu$.
  \item The quotient $Y^\la_\mu $ quantizes a (possibly
non-reduced) scheme supported on $ \Grlmbar$.
  \end{enumerate}
\end{itheorem}

Item (1) above is proven using the Drinfeld-Gavarini quantum groups duality, (2) follows simply from (1), and (3) follows using the GKLO representation.  In fact, we produce a family $ Y^\la_\mu(\mathbf c) $ of quantizations which we conjecture to map surjectively to the universal family in the sense of Bezrukavnikov-Kaledin \cite{BK04a}.

Unfortunately, we are not able to prove that the scheme quantized by $Y^\la_\mu $ is reduced.  However, we do provide a conjectural description of the generators of the ideal of $
\Grlmbar $ inside $ \Gr_\mu $ and prove that this conjecture implies that $Y^\la_\mu $ quantizes the reduced scheme structure on  $ \Grlmbar$.
Moreover, we prove that this conjecture gives a simple description for the ideal defining $ Y^\la_\mu$.

\subsection{Motivation and relation to other work}

Brundan-Kleshchev \cite{BK} construct an isomorphism between quotients
of shifted Yangians of $ \mathfrak{gl}_n $ and $ W$-algebras of $
\mathfrak{gl}_m $.  On one hand, it is known that $ W$-algebras are
quantizations of Slodowy slices.  On the other hand, by the work of
\Mirkovic -Vybornov  \cite{MVy} we have an isomorphism between Slodowy
slices for $ \mathfrak{sl}_m $ and slices in the affine Grassmannian
for $ GL_n $.  Thus via these results, we see that quotients of
shifted Yangians for $ \mathfrak{gl}_n $ quantize slices in the affine
Grassmannian for $ GL_n $.  This motivated us to look for a direct
construction of quantizations of affine Grassmannian slices (for any
semisimple $ G $) using quotients of shifted Yangians. (The idea that
the Brundan-Kleshchev isomorphism should be thought as a quantization
of the \Mirkovic -Vybornov isomorphism was independently observed by Losev \cite[Remark 5.3.4]{Lo}.)

If we take a limit of $ \Grlmbar$ as $ \lambda \rightarrow \infty $
and $ \lambda-\mu $ is fixed, then the slice $ \Grlmbar $ becomes the
Zastava space $ Z_{\lambda - \mu}$.  Finkelberg-Rybnikov \cite{FR} have
given conjectural quantizations of Zastava spaces (for $PGL_n$) using
quotients of Borel Yangians, which are a limit of shifted Yangians.
Thus in this limit we prove their conjectures, dependent on the above
mentioned conjecture about the ideal of $ \Grlmbar$.

Earlier work on shifted Yangians by Brundan and Kleshchev \cite{BK2}
suggest that one natural direction for future work is the study of a
version of category $\cO$ over the algebra $Y^\la_\mu$.  Because of the geometric Satake correspondence, we think of category $ \cO $ for $ Y^\la_\mu $ as a categorification of a weight space in a representation of the Langlands dual group $G^\vee$.  Thus we expect that these categories (with $\la $ fixed) carry categorical $ \fg^\vee $-actions.  Moreover, conjectures of Braden,
Licata, Proudfoot and the second author \cite{BLPW} suggest that category $ \cO $ for $ Y^\la_\mu $
should be Koszul dual to similar categories constructed from quiver varieties (in type A, we expect that this reduces to parabolic-singular duality of Beilinson-Ginzburg-Soergel \cite{BGS}).

\subsection{Acknowledgements}
We would like to thank Alexander Braverman, Pavel Etingof, Mikhail
Finkelberg, Ivan \Mirkovic, Sergey Oblezin, Travis Schedler and
Catharina Stroppel for extremely useful conversations.

\section{Symplectic structure on slices in the affine Grassmannian}

\subsection{Notation}\label{sec:notation}
For any group $ H $, we will write $ H((t^{-1})) = H(\C((t^{-1}))) $ for its loop group and write $ H[t] = H(\C[t]) $ and $ H[[t^{-1}]] = H(\C[[t^{-1}]]) $ for its usual subgroups.  Let $ H_1[[t^{-1}]] $ denote the first congruence subgroup of $ H[[t^{-1}]] $, i.e. the kernel of the evaluation at $ t^{-1} = 0 $, $ H[[t^{-1}]] \rightarrow H $.

Throughout $ G $ will denote a fixed complex semisimple group, with opposite Borel subgroups $ B, B_-$, unipotent subgroups $N, N_-$, maximal torus $ T$, coweight lattice $ X $, Weyl group $ W $, set of roots $ \Delta $, simple roots $ \{\alpha_i\}_{i\in I} $.  We write $ \{\om_i \}_{i \in I} $ for the fundamental weights of the simply connected form of $ G $.

Following Drinfeld, we use generators $e_i, f_i, h_i$ for $\fg$ where
$$ [h_i, e_j] = (\alpha_i, \alpha_j) e_j, \quad [h_i, f_j] = -(\alpha_i,\alpha_j) f_j, \quad [e_i, f_j] = \delta_{ij} h_i $$
along with the usual Serre relations.  Let $(a_{ij})_{1 \leq i,j \leq n}$ be the Cartan matrix of $\fg$, and let $d_i$ be the unique coprime positive integers such that $b_{ij} = d_i a_{ij}$ is a symmetric matrix.  Then the associated invariant form on $\fg$ is defined by $(e_i, f_j) = \delta_{ij}$, and $(\alpha_i, \alpha_j) = (h_i, h_j) = d_i a_{ij}$, and in particular $h_i$ is the image of $\alpha_i$ under the identification of $\fh$ and $\fh^\ast$.

This is as opposed to the standard Chevalley generators $e_i', f_i', h_i'$, which we will identify as
$$e_i = - d_i^{1/2} e_i',\quad f_i = - d_i^{1/2} f_i',\quad h_i = d_i h_i'$$
In this way we have fundamental weights $\omega_i(h_j') = \delta_{ij}$, and a lift of the Weyl group defined via $\overline{s_i} = \exp{(f_i')}\exp{(-e_i')}\exp{(f_i')}$.

If $ \mu $ is a weight or coweight, we write $ \mu^* = -w_0 \mu $.  Likewise, we write $ i^* $ if $\alpha_{i^*} = -w_0 \alpha_i $.

  Let $ V $ be a representation of $ G $, and let $ v \in V $,$ \beta \in W^*$.  The matrix entry $ \Delta_{\beta, v} $ is a function on $ G $ given by $ \Delta_{\beta,v}(g) = \langle
\beta, g v \rangle $.  If  $w_1,w_2 \in W $ and $ \tau $ a dominant weight, we define
$$ \Delta_{w_1\tau, w_2\tau}(g) = \langle\overline{w_1} v_{-\tau}, \overline{w_2} v_\tau\rangle $$
using the lift described above, where $v_\tau$ is the highest weight vector for the irreducible representation $V(\tau)$ and $ v_{-\tau} $ is the dual lowest weight vector in $ V(\tau^*) $.

Using this matrix entry (also known as generalized minor), we define the function $ \Delta_{\beta,v}^{(s)} $ on $G((t^{-1})) $, for $ s \in
\mathbb{Z} $, whose value at $ g $ is the coefficient of the
polynomial $ \Delta_{\beta, v}(g)$.  More precisely, these are given by
the formula
\begin{equation*}
\Delta_{\beta,v}(g) =  \sum_{s=-\infty}^{\infty} \Delta_{\beta, v}^{(s)}(g) t^{-s}
\end{equation*}

\subsection{Slices in the affine Grassmannian}
Let $ G $ be a semisimple complex group. In this paper, we will work with the \textbf{thick affine Grassmannian} $ \Gr
= G((t^{-1})) / G[t]  $.  We have an embedding of the usual thin affine Grassmannian into the thick affine Grassmannian
$$
G((t))/G[[t]] \cong G[t,t^{-1}]/G[t] \hookrightarrow G((t^{-1})) / G[t]
$$
In this paper, we work with the thick affine Grassmannian since it is
forced upon us by the non-commutative algebras we consider.  One
manifestation of this is the fact that the thick
Grassmannian is an honest scheme, while the thin Grassmannian is only
an ind-scheme.  However, at a first reading, this difference will be
of little importance, and the reader can pretend that we are working
with the usual thin affine Grassmannian.

Any coweight $ \lambda$ can be thought of as $\C[t,t^{-1}]$-point of
$G$, which we can think of as a $\C((t^{-1}))$-point as well.
To avoid confusion, we use $t^\lambda$ to denote this point in $ G((t^{-1})) $.  We also use $t^\lambda$ for the image of $ t^\la $ in $ \Gr $.

Let $\la$ and $\mu$ will denote
dominant weights.  Define \[ \Gr^\lambda =  \Gp t^\lambda, \qquad \Gr_{\mu} = \Gm t^{
  w_0 \mu}. \]   Recall that the thin affine Grassmannian is precisely $\cup_\lambda \Gr^\lambda $.

Our main object of interest
will be \[ \Grlmbar := \overline{\Gr^\lambda} \cap \Gr_\mu.\]  This
variety is a transverse slice to $\Gr^\mu $ inside of $
\overline{\Gr^\lambda}$ since $\Gr_\mu$ intersects every $\Gr^\nu$
transversely, and the intersection $\Gr^{\overline{\mu}}_\mu$ is just the
point $t^{w_0\mu}$.  In particular, this variety is non-empty if and only if $\mu\leq \la$,
that is, if $\Gr^\mu\subset \overline{\Gr^\la}$.  These varieties arise naturally under the geometric Satake correspondence of Lusztig \cite{L}, Ginzburg \cite{G}, and Mirkovi\'c-Vilonen \cite{MV}: the intersection homology of $\Grlmbar$ is identified with the $\mu$ weight space of the irreducible $G^\vee$-representation of highest weight $\la$.

Note that $ \kb^\times $ acts on $\Gr $ by loop rotation.  This action preserves the $ G[t] $ and $ \Gm $ orbits and so $ \kb^\times $ acts on $ \Grlmbar $.   The following result is standard.

\begin{proposition}\mbox{}
\begin{enumerate}
\item $ \Grlmbar $ is an affine variety of dimension $ 2\langle \rho, \lambda - \mu \rangle$.
\item The action of $ \kb^\times $ on $ \Grlmbar $ contracts $ \Grlmbar$ to the unique fixed point $ t^{w_0 \mu} $.
\end{enumerate}
\end{proposition}

\begin{example} \label{eg:Kleinian}
  If $\la=\mu+\al_i^\vee$, then $ \Grlmbar $ is isomorphic
  to the Kleinian singularity $\C^2/(\Z/n+2)$ where $ n= \langle \mu, \alpha_i \rangle$.
To see this, first we identify
  $$
  \C^2/(\Z/n+2) = \{ (u, v, w) : uv + w^{n+2} = 0 \}
  $$
  and then we define the isomorphism
  \begin{align*}
  \C^2/(\Z/n+2) &\rightarrow \Grlmbar \\
   (u,v,w) &\mapsto \phi_i \left( \left[ \begin{smallmatrix} 1 - wt^{-1}  & v t^{-(n+1)} \\ ut^{-1} & 1 + wt^{-1} + \dots + w^{n+1} t^{-(n+1)} \end{smallmatrix} \right] \right)t^{w_0 \mu}
  \end{align*}
  where $ \phi_i : SL_2 \rightarrow G $ denotes the $ SL_2 $ corresponding to $ \alpha_i $.
\end{example}

Let $ G((t^{-1}))_\mu $ denote the stabilizer of $ t^{w_0 \mu} $ inside of $ G((t^{-1}))$.   The following easy result describes the stabilizer on the Lie algebra level.
\begin{lemma} \label{le:LieStab}
 $Lie(G((t^{-1}))_\mu) = \mathfrak{t}[t] \oplus \bigoplus_{\alpha \in \Delta} t^{\langle \alpha, w_0 \mu \rangle} \fg_\alpha[t]$.
\end{lemma}
\begin{proof}
 The result follows immediately after observing that for $ g \in G((t^{-1}))$, we have $ g \in G((t^{-1}))_\mu $ if and only if $ t^{-w_0\mu} g t^{w_0 \mu} \in G[t]$.
\end{proof}

In what follows, we will need the following set-theoretic description
of $ \overline{\Gr^\lambda} $ due to Finkelberg-\Mirkovic \cite{FM}.
As we shall see, it is much trickier to find a description of this
variety with its natural reduced scheme structure.
\begin{proposition} \label{pr:SetTheoryGrlam}
Let $ g \in G((t^{-1})) $.  We have $ [g] \in \overline{\Gr^\lambda} $ if and only if $ \Delta_{\beta, v}^{(s)}(g) = 0 $ for all dominant weights $ \tau $, for all $ v \in V(\tau), \beta \in V(\tau)^* $ and for all $ s < \langle \lambda, w_0 \tau \rangle  $.
\end{proposition}
\begin{proof}
 Fix $ \tau $ and let $ k $ be the minimal $ s $ such that there exists $ \beta \in V(\tau)^*, v \in V(\tau) $ with $ \Delta_{\beta, v}^{(s)} (g) \ne 0 $ (if such a minimum exists).  It is easy to see that $ k $ only depends on the $G[t] $ double coset containing $ g $.  Thus if $ [g] \in \Gr^\lambda $, we have that $ k = \langle \lambda, w_0 \tau \rangle $. The result follows.
\end{proof}
The proof makes it clear that the Proposition holds even if $ \tau $ only ranges over a set of dominant weights which spans (over $ \mathbb Q $) the weight lattice.

\subsection{Symplectic structure on the affine Grassmannian} \label{se:SymplecticStructure}
There is a non-degenerate pairing on $\fg((t^{-1})) $ coming from residue
and the invariant form on $\fg $.  Hence the Lie algebras $ \fg[t]
$, $ t^{-1}\fg[[t^{-1}]] $, and $ \fg((t^{-1})) $ form a Manin triple (see
\cite{Dr}).  This induces a Poisson-Lie structure on $G((t^{-1}))$ with
$G[t]$ and $\Gm$ as Poisson subgroups.  In particular, it coinduces
a Poisson structure on $\Gr$, by standard calculations which date back
to work of Drinfeld \cite{Drpoisson}.

Let us state a couple of results concerning the interaction between
this symplectic structure and the geometry considered in the previous
section.  These results were originally obtained by \Mirkovic \cite{MPC}.

\begin{theorem}
The subvarieties $ \Gr^\lambda_\mu = Gr^\lambda \cap \Gr_\mu $ are symplectic leaves of $ \Gr$.
\end{theorem}
\begin{proof}
First we note that $ \Gr^\lambda_\mu $ are connected by
\cite[1.4]{Rich}, since $\fg((t^{-1}))=\fg[t]\oplus t^{-1}\fg[[t^{-1}]]$. The argument is stated there for finite dimensional
groups, but carries through to the loop situation without issues.  Then the result follows from \cite[Corollary 2.9]{LY}.
\end{proof}
These are not all symplectic leaves of $\Gr$, since not every $\Gm$-orbit
contains a point $t^{w_0\mu}$ and not every $ G[t] $ orbit contains a point $ t^\lambda $.  A general symplectic leaf which lies in the thin affine Grassmannian is of the form $ \Gr^\lambda \cap \Gm gt^{w_0\mu} $ where $ g \in G $.

Let $ S^\mu = N((t^{-1})) t^{w_0\mu}$. An {\bf MV cycle} is a component of $
\overline{\Gr^{\lambda}} \cap S^\mu $.  By \Mirkovic -Vilonen, these MV cycles give a basis for weight spaces of irreducible representations of the Langlands dual group.  As we now see the MV cycles are Lagrangians in $ \Grlmbar $.

\begin{proposition}
$\overline{\Gr^\lambda} \cap S^\mu $ is a Lagrangian subvariety of $ \Grlmbar $.
\end{proposition}
\begin{proof}
First we prove that $ \overline{\Gr^\la} \cap S^\mu \subset \Grlmbar $.  Since $ N$ is unipotent, we have that $ N((t^{-1})) =  N_1[[t^{-1}]] N[t]$.  Now by Lemma \ref{le:LieStab}, we have that $  N[t]t^{w_0 \mu} = t^{w_0 \mu} $.  Hence $N((t^{-1})) t^{w_0 \mu} = N_1[[t^{-1}]] t^{w_0 \mu} $ and thus $ S^\mu \subset \Gr_\mu $.

From \cite{MV}, $\dim \overline{\Gr^\lambda} \cap S^\mu = \langle \rho, \lambda-\mu \rangle $ and thus the intersection $\overline{\Gr^\lambda} \cap S^\mu $ is half-dimensional in $\Grlmbar$,
Hence it is Lagrangian if and only if it is coisotropic.  The variety
$\Grlmbar$ is affine, and so it suffices to check that the Poisson
bracket of any two functions that vanish on $\overline{\Gr^\lambda} \cap S^\mu$
vanishes there as well.  The functions vanishing on $S^\mu\cap
\overline{\Gr^\lambda}$ are generated by all functions of negative weight under the
action of the coweight $\rho^\vee:\C^\times \to G$.  Since that
action preserves the Poisson structure, the Poisson bracket of two
negative weight functions is again negative weight; this completes the proof.
\end{proof}

It is natural to ask whether $ \Grlmbar $ has a symplectic resolution.  Let us temporarily assume that $ G $ is of adjoint type and let us fix an sequence $\vlam= (\lambda_1, \dots, \lambda_n) $ of fundamental coweights such that $ \lambda = \lambda_1 + \cdots + \lambda_n $.  Then we have the open and closed convolutions
$$ \Gr^{\vlam} := Gr^{\lambda_1} \tilde{\times} \cdots \tilde{\times} \Gr^{\lambda_n}, \overline{\Gr^{\vlam}} := \overline{Gr^{\lambda_1}} \tilde{\times} \cdots \tilde{\times} \overline{\Gr^{\lambda_n}} $$
along with the convolution morphisms $ m :\Gr^{\vlam} \rightarrow \overline{\Gr^\lambda} $ and $ \bar{m} : \overline{\Gr^{\vlam}} \rightarrow \overline{\Gr^\lambda} $ .  Let \[ \Gr^{\vlam}_\mu := m^{-1}(\Gr_\mu) \qquad\Gr^{\overline{\vlam}}_\mu := \bar{m}^{-1}(\Gr_\mu) .\]

Recall that a normal variety $X$ with a fixed symplectic
structure $\Omega$
on its smooth locus is said to have
{\bf symplectic singularities} if, locally on $X$, there are
  resolutions of singularities $p\colon U\to X$ where $p^*\Omega$ is
  the restriction of a closed  2-form on $U$ (which is not
  assumed to be non-degenerate on the exceptional locus).

  A variety $X$ is said to have {\bf terminal singularities} if there is a
  resolution of singularities of $X$ such that each irreducible
  exceptional fiber has positive discrepancy, that is, $X$ is as close
  to being smoothly resolved as is crepantly possible.  A {\bf
    terminalization} $X\to Y$ is a map which is birational,
  proper, and crepant with $X$ having terminal singularities.  We say
  a variety $X$ is {\bf $\mathbb{Q}$-factorial} if every Weil divisor on $X$
  has an integer multiple which is Cartier.

\begin{theorem} \label{th:terminalization}
The variety $ \Grlmbar $ has symplectic singularities, and
$\Gr^{\overline{\vlam}}_\mu $ is a
$\mathbb{Q}$-factorial terminalization of $ \Grlmbar $.
\end{theorem}
\begin{proof}
First, we claim that $\Gr^{\overline{\vlam}}_\mu $ has singular locus
in codimension $\geq 4$.  Since $\Gr_\mu$ is transverse to every
$G[t]$-orbit, the codimension of the singular locus cannot jump when
we pass to $\Gr^{\overline{\vlam}}_\mu$, so we need only
establish the same result for $\Gr^{\overline{\vlam}}$, for which it
suffices to consider the case of a fundamental coweight.  If $\om_i$
is a fundamental coweight, and $\nu$ is a dominant coweight such that $\Gr^\nu\subset
\Gr^{\overline{\om_i}}$, then we have that $\rho^\vee(\om_i-\nu)\geq 2$,
since $\om_i-\al_j$ is never dominant.  Thus, the singular locus $\Gr^\nu$ has
codimension at least $4$.

As Beauville notes \cite[(1.2)]{Beau}, since  $\Gr^{\overline{\vlam}}_\mu $ is regular in
codimension 3 and normal, the existence of a symplectic form on its
smooth locus implies that it has symplectic singularities.  Since we
have a Poisson map $\Gr^{\overline{\vlam}}_\mu\to \Grlmbar $, this
variety also has symplectic singularities.  By a result of Namikawa \cite{Nanote},
this regularity in codimension 3 also implies that $\Gr^{\overline{\vlam}}_\mu $ is
terminal.

Since each local singularity in $\Gr^{\overline{\vlam}}_\mu $ is a
local singularity in $\Gr^{\overline{\vlam}}$, and these are the
product of local singularities in $\Gr^{\overline{\om_i}}$, we need only
prove $\mathbb{Q}$-factoriality in this case.  The group of Weil divisors of
$\Gr^{\overline{\om_i}}$ is the same as that of $\Gr^{{\om_i}}$ which is an
affine bundle over $G/P_i$ where $P_i$ is the maximal parabolic
containing all negative simple root spaces but $\mathfrak{g}_{-\al_i}$.  Thus,
the Weil divisor group of $G/P_i$ is isomorphic to $\Z$.

Since $\Gr^{\overline{\om_i}}$ is projective, {\it some} Weil divisor on
$\Gr^{\overline{\om_i}}$ is Cartier. Thus, the group
generated by any non-trivial Weil divisor must intersect the image of
the Cartier divisors, and so $\Gr^{\overline{\om_i}}$ is
$\mathbb{Q}$-factorial\footnote{We thank Alexander Braverman for
  suggesting this portion of the argument to us.}.

Furthermore, the map $\Gr^{\overline{\vlam}}_\mu \to \Grlmbar$
well-known to be proper and birational.  The preimage of $\Gr^\mu$ for
$\mu\neq \la, \la-\al_i$ has codimension $\geq 4$, so any exceptional
divisor must be the closure of a component of the preimage of
$\Gr^{\la-\al_i}$.  The coefficients of these divisors in the
discrepancy can thus be computed locally in a neighborhood of $x\in
\Gr^{\la-\al_i}$, but the germ of the map is equivalent to the minimal
resolution of a Kleinian singularity by Example \ref{eg:Kleinian}.  The Kleinian singularities are known to be crepant.
\end{proof}

An obvious question is when $ \Grlmbar $ has a symplectic resolution.  First, we make the following conjecture.

\begin{conjecture}
  Any symplectic resolution of $ \Grlmbar $ is of the form
  $\Gr^{\overline{\vlam}}_\mu$.
\end{conjecture}

We can easily see when $\Gr^{\overline{\vlam}}_\mu$ is actually a resolution.

\begin{theorem}
The following are equivalent.
\begin{enumerate}
\item $ \Grlmbar$ possesses a symplectic resolution of singularities.
\item $\Gr^{\overline{\vlam}}_\mu$ is smooth and thus is a symplectic
  resolution of singularities of $ \Grlmbar$.
\item $ \Gr^{\vlam}_\mu = \Gr^{\overline{\vlam}}_\mu$.
\item There do not exist coweights $ \nu_1, \dots, \nu_n $ such that $ \nu_1 + \cdots + \nu_n = \mu $,  for all $k $, $ \nu_k $ is a weight of $ V(\lambda_k) $ and for some $ k $, $ \nu_k $ is a not an extremal weight of $ V(\lambda_k) $.
\end{enumerate}
\end{theorem}
\begin{proof}
\noindent (i) $\Rightarrow$ (ii): If $ \Grlmbar$ has a symplectic
resolution, then by \cite[5.6]{NaP} any $\mathbb{Q}$-factorial
terminalization of $ \Grlmbar$, in particular
$\Gr^{\overline{\vlam}}_\mu$, is smooth.

\noindent (ii) $\Rightarrow$ (i): In this case,
$\Gr^{\overline{\vlam}}_\mu$ is an example of  a symplectic resolution of singularities.

\noindent (ii) $\Rightarrow$ (iii):
It is well-known that the smooth locus of $ \overline{\Gr^\lambda} $ is precisely $ \Gr^\lambda $.  Thus the smooth locus of $ \Gr^{\overline{\vlam}} $ is precisely $ \Gr^{\vlam} $.

Now assume that there is a point $x$ in $\Gr^{\overline{\vlam}}_\mu $ not in $\Gr^{\vlam}_\mu$;  we know that $\Gr^{\overline{\vlam}}$ is not smooth at $x$.  By the transversality of the $ G_1[[t^{-1}]] $ and $ G[t] $ orbits, the completion of  $\Gr^{\overline{\vlam}}$  at $x$ is the same as the completion of $\Gr^{\overline{\vlam}}_\mu $ at $x$ times something smooth.  Therefore $ \Gr^{\overline{\vlam}}_\mu$ cannot be smooth at $x$ either.

\noindent (iii) $\Rightarrow$ (ii): clear.

\noindent (iii) $\Rightarrow$ (iv):
If there exist $ \nu_1, \dots, \nu_n $ as in (iii), then $ (t^{\nu_1}, t^{\nu_1 + \nu_2}, \dots, t^{\mu}) \in \Gr^{\overline{\vlam}}_\mu \smallsetminus \Gr^{\vlam}_\mu$.

\noindent (iv) $\Rightarrow$ (iii): Suppose that there exists
$$ (L_1, \dots, L_n) \in \Gr^{\overline{\vlam}}_\mu \smallsetminus \Gr^{\vlam}_\mu $$
Recall that we have a $ \C^\times \times T $ action on $ \Gr $ where the first factor acts by loop rotation.  Consider a map $ \C^\times \rightarrow \C^\times \times T $ which is the identity into the first factor and a generic dominant coweight into the second factor.  We get a resulting $ \C^\times $ action on $ \Gr $ whose attracting sets are the $ I_- $ orbits, where $ I_- $ is the preimage of $ B $ under $ G[[t^{-1}]] \rightarrow G $.

Let $ (t^{\mu_1}, \dots, t^{\mu_n}) = \lim_{s \rightarrow 0} s \cdot (L_1, \dots, L_n) $.  From the definition of $\Gr^{\overline{\vlam}}_\mu$, we see that $ \mu_n = \mu $.  Also, we see that for each $k $, $ d(t^{\mu_{k-1}}, t^{\mu_k}) \le \lambda_k $ (where $ d : \Gr \times \Gr \rightarrow X_+ $ denotes the $ G((t^{-1})) $-invariant distance function on $\Gr $) and so $ \nu_k := \mu_k - \mu_{k-1} $ is a weight of $ V(\lambda_k)$.  Thus we obtain $ \nu_1, \dots, \nu_n $ with $ \nu_1 + \cdots + \nu_n = \mu $.  Moreover, since $ (L_1, \dots, L_n) \notin \Gr^{\vlam}_\mu $, for some $ k $, $ d(L_{k-1}, L_k) < \lambda_k $ and so $ \nu_k $ is a non-extremal weight of $ V(\lambda_k) $.
\end{proof}

If  $ \lambda $ is a sum of minuscule coweights, then the above conditions hold.  For any simple
$G $ not of type A, there are non-miniscule fundamental
coweights $ \lambda $; for such $ \lambda $, we can choose $ \mu $ such that the above conditions do not hold.  So there exist $ \Grlmbar $ which do not admit symplectic resolutions.

\subsection{Beilinson-Drinfeld Grassmannian}
Using the Beilinson-Drinfeld Grassmannian, we can define a family of
Poisson varieties over $\mathbb{A}^n $ whose special fibre is $
\Grlmbar $.  In this work, this family will only be used as motivation
for a similar family of quantizations of $ \Grlmbar $; as illustrated
in works such as \cite{BK04a, BPW, Lo}, the universal symplectic deformation of a
symplectic singularity as a symplectic variety is intimately tied to
understanding its quantizations (see section \ref{se:universality}). From this perspective, a
natural next step (beyond the scope of this paper) would be to study quantizations of the total spaces
of these deformations, not just of a single fibre.

Recall that we have the moduli interpretation of the affine Grassmannian
\begin{multline*}
\Gr = \{ (E, \phi) : E \text{ is a principal $G$-bundle on $ \pp $ and}\\\text{ $\phi: E |_{\pp \smallsetminus \{0 \}} \rightarrow E^0|_{\pp \smallsetminus \{0\}}  $ is an isomorphism} \}
\end{multline*}
where $ E^0 $ denotes the trivial $ G $-bundle.  We say that $ (E, \phi) $ has \textbf{Hecke type} $ \lambda $ at $0 $ if $ (E, \phi) $ gives a point in $ \Gr^\lambda $ under the above identification.

Note that the action of $ G[[t^{-1}]] $ by left multiplication in the homogeneous space definition becomes change of trivialization in the new definition.  Thus the $ G[[t^{-1}]] $ orbit of $ (E, \phi) $ is determined by isomorphism class of the $ G $-bundle $E $, which is given by a dominant coweight.  Note also that the action of $\Gm $ corresponds to changes of trivialization which do not change anything at $ \infty$.

Let $ \mu $ be a dominant weight and let $ P $ be the corresponding standard parabolic subgroup (so that $ W_P $ is the stabilizer of $ \mu $ in the Weyl group).  Let $ E $ be a principal $ G$-bundle of type $ \mu $.  Then $ E $ has a canonical $ P $-structure.

Now let $ (E, \phi) \in \Gr $.  Let $ \mu $ be the isomorphism type of $ E $.  Then $ \phi_\infty $ carries the parabolic structure at $ \infty $ to a parabolic subgroup of $ G $ of type $ \mu $.  Hence we see that the $ \Gm $ orbits on $ \Gr$  are labelled by a pair consisting of a dominant weight $\mu $ and a parabolic subgroup of $ G $ of type $ \mu $.   In particular $ \Gr_\mu $ is the locus of those $ (E, \phi) $ where $ E $ has isomorphism type $ \mu $ and the parabolic subgroup produced is the standard one.

We now will consider the Beilinson-Drinfeld deformation of the affine Grassmannian.  This is a family $ \Gr_{\mathbb{A}^n} $ over $ \mathbb{A}^n $ whose fibre at $ a_1, \dots, a_n \in \mathbb{A}^n $ is given as follows:
\begin{equation*}
\begin{aligned}
\Gr_{a_1, \dots, a_n} = \{ (E, \phi) :&E \text{ is a principal $G$-bundle on $\pp$ and} \\
&\phi: E|_{\pp \smallsetminus \{a_1, \dots, a_n \} } \rightarrow E^0|_{\pp \smallsetminus \{a_1, \dots, a_n\}} \text{ is an isomorphism } \}
\end{aligned}
\end{equation*}

Let $\Gr_{\mu, \mathbb{A}^n} $ be the locus of $ (E, \phi) $ where $ E $ has isomorphism type $ \mu $ and the parabolic subgroup at $ \infty$ is the standard one.

Specializing to one choice of parameters, we can consider changes of trivialization acting on $ \Gr_{a_1, \dots, a_n} $.  Let $ G_1(\pp \smallsetminus \{a_1, \dots, a_n \}) $ denote the kernel of $ G(\pp \smallsetminus \{a_1, \dots, a_n \}) \rightarrow G $ given by evaluation at $ \infty $.  Then, $ \Gr_{\mu, (a_1, \dots, a_n)} $ is an orbit of $ G_1(\pp\smallsetminus \{a_1, \dots, a_n \} )$.

We may also think of this locus in terms of the $ \kb^\times $ action.  We have an action of $ \kb^\times $ on $ \Gr_{\mathbb{A}^n} $ coming from the action of $ \kb^\times $ on $ \pp$.  Note that this action moves the base $ \mathbb{A}^n$.  On the central fibre $ \Gr_{(0, \dots, 0)} = \Gr $ this action of $ \kb^\times $ restricts to the loop rotation action on $ \Gr $.  Hence the fixed points of this $ \kb^\times $ action are the same as the fixed points of the loop rotation action, namely the sets $ G t^\mu $ inside the affine Grassmannian.
Moreover, we have that $ \Gr_{\mu,\mathbb{A}^n} $ is the attracting set for $ t^{w_0 \mu} $ under
the $ \kb^\times $ action.

We have a fiberwise Poisson structure on $ \Gr_{\mathbb{A}^n}$ using the Manin triples described in Etingof-Kazhdan \cite{EK}, Corollary 2.10 and Proposition 2.12.  As in Section \ref{se:SymplecticStructure}, we get a Poisson structure on $\Gr_{\mu, (a_1, \dots, a_n)}$.

Now, let us choose an expression $ \lambda = \lambda_1 + \cdots + \lambda_n $, where $ \lambda_1, \dots, \lambda_n $ are fundamental coweights.  This gives us an $ X_+ $ colored divisor $ D $ on $ \pp $ defined by $ D = \sum \lambda_i a_i $.  We will think of $ D $ as a function $ \pp \rightarrow X_+$.  Now we define
\begin{equation*}
\Gr_{\mu, (a_1, \dots, a_n)}^{\lambda_1, \dots, \lambda_n} \\:= \{ (E, \phi) \in \Gr_{\mu, (a_1, \dots, a_n)} : (E,\phi) \text{ has Hecke type } D(x)\text{ for all } x\in \pp \}
\end{equation*}
 From the above analysis, it is possible to show that these are symplectic leaves in $ \Gr_{\mu, (a_1, \dots, a_n)}$.

 Fixing $ (\lambda_1, \dots, \lambda_n) $ and letting $ (a_1, \dots, a_n) $ vary, this forms a family of $ \mathbb{A}^n $.  The central fibre of this family is $ \Gr_\mu^{\lambda} $.

 Now, define
 \begin{equation*}
 \Gr_{\mu, (a_1, \dots, a_n)}^{\overline{\lambda_1, \dots, \lambda_n}}\\ := \{ (E, \phi) \in \Gr_{\mu, (a_1, \dots, a_n)} : (E,\phi) \text{ has Hecke type } \le D(x) \text{ } \text{ for all } x\in \pp \}
 \end{equation*}
 Then we obtain a flat family of symplectic varieties over $ \mathbb{A}^n $ whose central fibre is $ \Grlmbar $.

 \subsection{Direct system on slices and Zastava spaces}
 We will now look at what happens to $ \Grlmbar $ when we increase $ \lambda, \mu $, keeping $ \lambda - \mu $ fixed.

 Let us fix $ \nu $ in the positive coroot cone.  Let $ \mu, \mu' $ be dominant coweights with $ \mu' - \mu $ dominant.  From Lemma \ref{le:LieStab}, the stabilizer of $ t^{w_0\mu'} $ in $\Gm $ contains the stabilizer of $ t^{w_0\mu} $ in $\Gm $.  So we can define a map $ \Gr_\mu \rightarrow \Gr_{\mu'}$ by $ gt^{w_0\mu} \mapsto gt^{w_0 \mu'} $.  From Proposition \ref{pr:SetTheoryGrlam}, we see that this  restricts to a map $ \Gr_\mu^{\overline{\mu + \nu}} \rightarrow \Gr_{\mu'}^{\overline{\mu' + \nu}}$.  By construction, it is a Poisson map.

 Clearly these maps are compatible with composition.  Thus with $ \nu $ fixed we get a direct system of slices $ \left\{ \Gr_\mu^{\overline{\mu + \nu}} \right\}_\mu $. The limit of this system is an ind-scheme, but in general it will not be represented by a scheme.

 On the other hand, we can consider the Zastava space $ Z_\nu $, an affine variety, as defined in \cite{FM}.  It is a compactification of the moduli space $ Z_\nu^{\circ} $ of based maps from $ \mathbb{P}^1$ into $ G/B $ of degree $ \nu $.    The variety $ Z_\nu $ carries an action of $ \kb^\times $, extending the action of $ \kb^\times $ on $Z_\nu^{\circ} $ which rotates the source of the map.

 The following result is Theorem 2.8 from Braverman-Finkelberg \cite{BF}.  It shows that the algebras of functions $ \O(\Gr_\mu^{\overline{\mu + \nu}}) $ stabilize to $ \O(Z_\nu)$.
 \begin{theorem} \label{th:maptoZastava}
 There exists a map $ \Gr_\mu^{\overline{\mu + \nu}} \rightarrow Z_\nu $.   These maps are compatible with the above direct system on the slices and with the actions of $ \mathbb{C}^\times $.  Moreover, the induced maps $ \O(Z_\nu)_N \rightarrow \O(\Gr_\mu^{\overline{\mu+\nu}})_N $ are isomorphisms if $ N \le \langle \alpha_i, \mu \rangle$ for all $ i $.
 \end{theorem}

 \begin{remark}
 The theorem provides $ Z_\nu $ with a Poisson structure.  On the other hand, $Z_\nu^{\circ} $ carries a symplectic structure as described in \cite{FKMM}.  It is expected that these two structures are compatible.
 \end{remark}

 \begin{example} \label{ex:1}
 Let us take $ G = PGL_2 $ and $ \nu = \alpha^\vee $, the simple coroot.
 Then (as in Example \ref{eg:Kleinian}), for $ n \ge 0$,
 $$
 \Gr_{n \omega^\vee }^{\overline{n\omega^\vee + \alpha^\vee}} \cong \{(u,v,w) : uv + w^{n+2} = 0 \}
 $$
 Moreover, for $ m \ge n $, the map $ \Gr_{n \omega^\vee}^{\overline{n\omega^\vee + \alpha^\vee}} \rightarrow \Gr_{m \omega^\vee}^{\overline{m\omega^\vee + \alpha^\vee}} $ is given by $ (u,v,w) \mapsto (u,vw^{m-n}, w) $.  This is because we have an equality in $ \Gr_{PGL_2} $
 \begin{equation*}
 \left[ \begin{smallmatrix} 1 - wt^{-1} & v t^{-(n+1)} \\ ut^{-1} & 1 + wt^{-1} + \dots + w^{n+1} t^{-(n+1)} \end{smallmatrix} \right]
 \left[\begin{smallmatrix} 1 & 0 \\ 0 & t^m \end{smallmatrix} \right]
 =
 \left[ \begin{smallmatrix} 1 - wt^{-1}  & vw^{m-n} t^{-(m+1)} \\ ut^{-1} & 1 + wt^{-1} + \dots + w^{m+1} t^{-(m+1)} \end{smallmatrix} \right]
 \left[\begin{smallmatrix} 1 & 0 \\ 0 & t^m\end{smallmatrix} \right].
 \end{equation*}
 On the other hand, the Zastava space $ Z_\alpha $ is $ \mathbb{A}^2$.  The map in Theorem \ref{th:maptoZastava} is given by $ (u,v,w) \mapsto (u,w) $.

 With respect to the $ \kb^\times $ action on
 $$
 \Gr_{n \omega^\vee}^{\overline{n\omega^\vee + \alpha^\vee}} = \{(u,v,w) : uv + w^{n+2} =0 \}
 $$
 the variables $ u, w $ have weight 1 and $ v $ has weight $ n+1 $.
 So, we can see that
 $$
 \O(Z_\alpha) = \kb[u,w] \rightarrow \O(\Gr_{n \omega^\vee}^{\overline{n\omega^\vee + \alpha^\vee}}) = \kb[u,v,w]/ (uv+w^{n+2})
 $$
 is an isomorphism in degrees $ 0, \dots, n $ as predicted by Theorem \ref{th:maptoZastava}.

 The Poisson structure on $ \Gr_{n \omega^\vee}^{\overline{n\omega^\vee + \alpha^\vee}} $ is given by
 $$
 \{ w, u \} =  u \quad \{w, v \} = - v \quad \{u, v \} = (n+2)w^{n+1}
 $$
 while the Poisson structure on $ Z_\alpha $ is given by
 $$
 \{ w, u \} = u.
 $$

 Finally, note that the $ \kb$-points of the ind-scheme $ \lim_n \Gr_{n \omega^\vee}^{\overline{n\omega^\vee + \alpha^\vee}} $ are
 $$ \{(u,w) : u \in \kb^{\times}, w \in \kb \} \cup \{(0,0) \} $$
 which is a proper subset of $ \kb^2 $ and hence this ind-scheme is not equal to $ \mathbb{A}^2 $.
 \end{example}

 \subsection{Description of the Poisson structure} \label{sec:Poisson structure}
 We would like to describe the Poisson structure on $ \Gm $ in a little
 more detail.
 Let $C \in \fg \otimes \fg $ be the Casimir element for the bilinear form. Picking dual bases, we may represent this element as $ C = \sum J_a
 \otimes J^a $; this Casimir element allows us to describe the Poisson
 bracket of two minors.  This can be written more compactly using the
 series $ \Delta_{\beta, v}(u)=\sum_{s\geq 0}  \Delta_{\beta, v}^{(s)}u^{-s}$.  Note that $ \Delta^{(0)}_{\beta, v} = \langle \beta, w \rangle $ is a constant function.

 \begin{proposition} \label{th:Poissonminor}
 In $ \O(\Gm)[[u_1^{-1}, u_2^{-1}]]$, the Poisson bracket $\{ \Delta_{\beta_1, v_1}(u_1), \Delta_{\beta_2,v_2}(u_2) \}$ is equal to
 \begin{equation*}
   \frac{1}{u_1-u_2} \sum_a \Delta_{ \beta_1,J_a v_1}(u_1) \Delta_{
     \beta_2,J^a v_2}(u_2) - \Delta_{J_a\beta_1, v_1}(u_1)
   \Delta_{J^a\beta_2, v_2 }(u_2)
 \end{equation*}
 \end{proposition}

 \begin{proof}
 The cobracket $\fg((t^{-1}))\to \fg((u_1))\otimes \fg((u_2))$ is coboundary.  If we let $r(u_1,u_2)=\frac{C}{u_1-u_2}$ it is given by \[a(t)\mapsto \left[a(u_1)\otimes 1+1\otimes a(u_2),r(u_1,u_2)\right].\]

 As described earlier, the Lie algebra $\fg((t^{-1}))$ carries an inner
 product $(f,g)_t=- \operatorname{res}_{t=0} (f,g)$ for which $\gm$ is
 Lagrangian and complementary to $\gp$; this realizes $\fg((t^{-1}))$ as the
 (topological) Drinfeld double of $\gm$.  In particular, $\Gm\subset
 G((t^{-1}))$ is a Poisson subgroup, and the Poisson bracket of any two
 functions on $\Gm$ can be calculated taking the bracket of any two
 extensions to all of $G((t^{-1}))$ and then restricting to $\Gm$.

 Thus, the Poisson structure on $G((t^{-1}))$ is defined by
 $r^L(u_1,u_2)-r^R(u_1,u_2)$, the difference of the left translation and right
 translation of the element $r(u_1,u_2)$ considered as a bivector at the
 identity.
 If $ X \in \gm $, and $ g \in \Gm$, we identify $ X $ with a tangent vector at $ g $ by left translation.  Then we have $$ (d \Delta_{\beta, v})_g(X) = \langle \beta, gXv \rangle. $$

 Hence
 \begin{multline*}
 \{ \Delta_{\beta_1,v_1}(u_1), \Delta_{\beta_2, v_2}(u_2) \}(g) = (d \Delta_{\beta_1, v_1})_g \otimes (d \Delta_{\beta_2, v_2})_g (g) \\
 = \frac{1}{u_1-u_2} \Big( \sum_a \langle \beta_1, g(u_1) J_a v_1 \rangle \langle \beta_2, g(u_2) J^a v_2 \rangle
 - \sum_a \langle \beta_1, J_a g(u_1)v_1 \rangle \langle \beta_2, J^a g(u_2) v_2\rangle \Big)
 \end{multline*}
 and then the proposition follows from the invariant of the pairing between dual representations.
 \end{proof}

 We can unpack Proposition \ref{th:Poissonminor} into the following equations:

 \begin{equation}\label{eq:minor-bracket}
 \{ \Delta_{ \beta_1, v_1}^{(r+1)}, \Delta_{\beta_2, v_2}^{(s)} \} - \{
 \Delta_{\beta_1, v_1}^{(r)}, \Delta_{ \beta_2, v_2}^{(s+1)} \} = \sum
 \Delta_{J_a\beta_1,  v_1}^{(r)} \Delta_{J^a\beta_2,  v_2}^{(s)} -
 \Delta_{ \beta_1, J_a v_1}^{(r)} \Delta_{\beta_2, J^a v_2}^{(s)}
 \end{equation}
 for $ r, s \ge 0$.  These equations specify all the desired Poisson brackets.

  \subsection{A conjectural description of the ideal of \texorpdfstring{$\Grlmbar$}{Gr}}
 In this section, we describe a conjectural description of the ideal of $ \Gr^{\overline{\lambda}}_\mu $ as a subvariety of $ \Gr_0 = \Gm$.  Let $ G_{sc} $ denote the simply connected cover of $ G $.  Note that the natural map $ G_{sc}[[t^{-1}]]_1 \rightarrow G[[t^{-1}]]_1 $ is an isomorphism.  This allows us to consider $ \Delta_{\om_i, \om_i}^{(s)} $ as functions on $\Gm $, even if $ \om_i $ are not weights of $ G $ (for example if $ G $ is of adjoint type).

We begin with the case of $ \mu = 0 $.  Let $J^\lambda_0 $ denote the ideal in $ \O(\Gm) $ Poisson generated by $ \Delta_{\om_i, \om_i}^{(s)} $ for $ s > \langle \lambda, \om_{i^*} \rangle$ and for $ i \in I $.
 \begin{conjecture} \label{co:MainConj}
  The ideal of $ \Gr^{\overline{\lambda}}_0 $ in $\O(\Gm) $ is $J^\lambda_0$.
 \end{conjecture}

 Let us make some comments on this conjecture.  First, we have the following result.

 \begin{proposition} \label{pr:J0gen}
  $J^\lambda_0$ is generated as an ordinary ideal by $ \Delta_{\beta, v}^{(s)} $ for $ s > \langle \lambda, \om_{i^*} \rangle$ and for $ i \in I $ and where $ \beta, v $ range over bases for $ V(\om_i)^* $ and $ V(\om_i)$.
 \end{proposition}

  \begin{proof}
 Let $I$ be the ideal  generated as an ordinary ideal by
 $ \Delta_{\beta, v}^{(s)} $ for $ s > \langle \lambda, \om_{i^*}
 \rangle$.  First, we show that this ideal is contained in $J^\lambda_0
 $.

 First, we claim that $ \Delta^{(s)}_{\om_i, v} \in J^\lambda_0 $ for all $ v \in V(\om_i) $ and $ s > \langle \la, \om_{i^*} \rangle $.  We proceed by downward induction on the weight of $ v $.  The base case of  $ v $ is highest weight follows by definition.  For the inductive step, suppose that $v$ is not highest weight
 weight.  In this case, $v=\sum f_j v_j$ for some $v_j$ of higher weight than $ v $.

 Fix $ s $ with $ s > \langle \la, \om_{i^*} \rangle $.  Using \eqref{eq:minor-bracket} and the
 expression for the Casimir (for notation see Section \ref{PBW Yangian})  \begin{equation*}
 C = C_\fh + \sum_{\alpha \in \Phi_+} C_\alpha e_\alpha \otimes f_\alpha + C_\alpha f_\alpha \otimes e_\alpha,
 \end{equation*} where $(e_\alpha, f_\alpha) = C_\alpha^{-1}$, we see that
 \[
 \{\Delta^{(s)}_{\om_i,v_j},\Delta^{(1)}_{\om_j,s_j\om_j}\}=-\Delta^{(s)}_{\beta,f_j v_j}\]
 Thus we see that
 $$\Delta^{(s)}_{\om_i, v} = \sum_j \Delta^{(s)}_{\om_i, f_jv_j} = - \sum_j \{ \Delta^{(s)}_{\om_i, v_j}, \Delta^{(1)}_{\om_j, s_j \om_j} \}.
 $$
 All the terms on the right hand side lie in $J^\lambda_0$ by the inductive assumption,
 and thus $ \Delta^{(s)}_{\om_i, v} \in J^\lambda_0 $.

 Now we claim that $  \Delta^{(s)}_{\beta, v} \in J^\lambda_0 $ for all $ \beta \in V(\om_i)^*, v \in V(\om_i)$ and $ s > \langle \la, \om_{i^*} \rangle $.

 We have already proven this claim when $ \beta = v_{-\om_i} $, so we proceed by induction on the weight of $ \beta$.  Suppose that $ \beta \in V(\om_i)^* $ is not lowest weight and assume that the claim holds for all $ \beta $ of lower weight.   In this case, we can write $ \beta = \sum e_j \beta_j $ for some $ \beta_j $ of lower weight.

 Fix $ s $ with $ s > \langle \la, \om_{i^*} \rangle $.  Again using the above expression for the Casimir we find that
 \[
 \{\Delta^{(s)}_{\beta_j,v},\Delta^{(1)}_{s_j\om_j,\om_j}\}=\Delta^{(s)}_{\beta_j,e_j v}-\Delta^{(s)}_{e_j\beta_j,v},\]

 Thus we see that
 $$\Delta^{(s)}_{\beta, v} = \sum_j \Delta^{(s)}_{e_j \beta_j,v} = \sum_j \{ \Delta^{(s)}_{\beta_j, v}, \Delta^{(1)}_{s_j \om_j, \om_j} \} - \Delta^{(s)}_{\beta_j,e_j v}.
 $$
 All the terms on the right hand side lie in $J^\lambda_0$ by the inductive assumption,
 and thus $ \Delta^{(s)}_{\beta, v} \in J^\lambda_0 $.
 This shows that $I\subset J^0_\la$.

 It remains to show that $ I $ is a Poisson ideal.  Since $ \Delta^{(s)}_{\beta, v} $, for $ \beta \in V(\om_i)^*, v \in V(\om_i), i \in I$, generates $ \O(G_1[[t^{-1}]]) $, it suffices to check that $ I $ is closed under Poisson bracket with these elements.  This follows immediately from \eqref{eq:minor-bracket}.
  \end{proof}

 Combining this proposition with Proposition \ref{pr:SetTheoryGrlam}, we obtain the following.

 \begin{corollary} \label{co:J0vanish}
 The vanishing set of $J^\lambda_0 $ is $ \Gr^{\overline{\lambda}}_0 $.
 \end{corollary}
 Thus in order to establish Conjecture \ref{co:MainConj}, it only remains to show that $I^\lambda_0$ is radical.

 \begin{remark}
  Let $ G = SL_n $.  By an observation which goes back to Lusztig, we know that there is an isomorphism $ \Gr^{\overline{n\omega_1}}_0 \cong \mathcal{N} $, the nilpotent cone of $ \mathfrak{sl}_n $.  For any dominant coweight $ \lambda $ with $ \lambda \le n \omega_1$, under this isomorphism $ \Gr^{\overline{\lambda}}_0 $ is taken to a nilpotent orbit closure.
 Thus, the above conjecture implies generators for the ideal of a
 nilpotent orbit closure inside the nilpotent cone of $
 \mathfrak{sl}_n $.  From this perspective, one can see that
 Conjecture \ref{co:MainConj} would imply the main result of Weyman
 \cite{W}, which gives generators for the ideals of nilpotent orbit closures.  This gives
 additional evidence toward the conjecture, but also suggests it will be difficult to prove.
  \end{remark}

 \begin{remark}
  One could imagine a similar conjecture for the ideal of $ \Gr^{\bar \lambda} $ inside of the homogeneous coordinate ring of $ \Gr $.  However, this conjecture is false, already for $ SL_2 $ and $ \lambda = \alpha $.
 \end{remark}

 We will need the following generalization of Conjecture \ref{co:MainConj} which describes the ideal of $ \Grlmbar$.  Consider the subgroup $ \Gm_\mu $ defined as the stabilizer in $ \Gm $ of $ t^{w_0 \mu} $.  Note that Lemma \ref{le:LieStab}, $\Gm_\mu \subset N_1[[t^{-1}]]$.

 By the orbit-stabilizer theorem, we see that $\Gr_\mu =  \Gm / \Gm_\mu $ and so $ \O(\Gr_\mu) = \O(\Gm)^{\Gm_\mu} $.  Moreover the map $ \Gm \rightarrow \Gr_\mu $ is Poisson and thus $ \O(\Gr_\mu) $ is a Poisson subalgebra of $ \O(\Gm) $.

 \begin{lemma} \label{le:GrmuMinors}
 The subalgebra $ \O(\Gr_\mu) $ contains
 \begin{gather*}
 \Delta^{(s)}_{s_i\om_i, \om_i}, \text{ for all } i \in I, s > 0, \quad \Delta^{(s)}_{\om_i, \om_i}, \text{ for all } i \in I, s > 0, \\
 (\Delta_{\om_i, s_i\om_i}/\Delta_{\om_i, \om_i})^{(s)}, \text{ for all } i \in I, s > \langle \mu^*, \alpha_i \rangle
 \end{gather*}
 \end{lemma}

 Later we will see that these elements  generate $ \O(\Gr_\mu) $ as a Poisson algebra.
 \begin{proof}
 Note that the action of $ \Gm_\mu $ on $ \O(\Gm) $ is given by $ (k \cdot f)(g) = g(fk) $ for $ k \in \Gm_\mu$, $ f \in \O(\Gm) $ and $ g \in \Gm $.  In particular, we see that $ k \cdot \Delta_{\beta, v} = \Delta_{\beta, kv} $.

 Since $ \Gm_\mu \subset N_1[[t^{-1}]]$, the minors $ \Delta_{\om_i, \om_i} $ and $ \Delta_{s_i \om_i, \om_i} $ will be $ \Gm_\mu$-invariant.  Hence all $\Delta^{(s)}_{s_i\om_i, \om_i}, \Delta^{(s)}_{\om_i, \om_i} $ all lie in $ \Gr_\mu $.

 On the other hand, let us consider the coefficients of the $\Delta_{\om_i, s_i \om_i} $ minor.  If $k\in \Gm_\mu$, then we have $k\cdot v_{s_i \om_i}=v_{s_i \om_i}+\Delta_{\om_i,s_i \om_i}(k)v_{\om_i}$.  Hence if $ g \in \Gm $, then
 \begin{equation*}
 \frac{\Delta_{\om_i,s_i\om_i}(gk)}{\Delta_{\om_i,\om_i}(gk)}
 =\frac{\Delta_{\om_i,s_i\om_i}(g)+\Delta_{\om_i,\om_i}(g)\Delta_{\om_i,s_i \om_i}(k)}{\Delta_{\om_i,\om_i}(g)}
 =\frac{\Delta_{\om_i,s_i\om_i}(g)}{\Delta_{\om_i,\om_i}(g)}+\Delta_{\om_i,s_i \om_i}(k)
 \end{equation*}
 By Lemma \ref{le:LieStab}, we have $\val \Delta_{\om_i,s_i\om_i}(k)\ge \langle w_0\mu, \alpha_i \rangle $.  Hence the coefficient of $t^{-s} $ in  $ \Delta_{\om_i, s_i \om_i}/\Delta_{\om_i, \om_i} $ is invariant under the action of $ \Gm_\mu$ for $ s > \langle \mu^*, \alpha_i \rangle $.  Thus $ (\Delta_{\om_i, s_i\om_i}/\Delta_{\om_i, \om_i})^{(s)} \in \O(\Gr_\mu) $ for $ s > \langle \mu^*, \alpha_i \rangle$.
 \end{proof}

 Let $ J_\mu^\lambda $ denote the ideal of $\O(\Gr_\mu) $  Poisson generated by $ \Delta_{\omega_i, \omega_i}^{(s)} $ for $ i \in I $ and $ s > \langle \lambda - \mu, \om_{i^*} \rangle =m_i$.
 \begin{conjecture} \label{co:main2}
 The ideal of $ \Grlmbar $ in $\O(\Gr_\mu) $ is $ J_\mu^\lambda $.
 \end{conjecture}

This conjecture generalizes Conjecture \ref{co:MainConj}.  When $ \mu
\ne 0 $, we do not have a set of (ordinary) generators for $
J_\mu^\lambda $ as in Proposition \ref{pr:J0gen}.  However, we will
now establish an analog of Corollary \ref{co:J0vanish}.

 \begin{proposition}\label{pr:GenJmulam}
   The vanishing locus of $ J_\mu^\lambda $ is $ \Grlmbar $.
 \end{proposition}
 \begin{proof}
   The vanishing locus of  $ J_\mu^\lambda $ is the union of the
   symplectic leaves  in the vanishing locus of $ \Delta_{\omega_i,
     \omega_i}^{(s)} $ for $ i \in I $ and $ s > \langle \lambda - \mu,
   \om_{i^*} \rangle =m_i$; after all, the vanishing set is a union of
   symplectic leaves and if these functions vanish on a
   symplectic leaf, so do all Poisson brackets with them.

 These generalized minors vanish on $\Grlmbar $ by Proposition
 \ref{pr:SetTheoryGrlam}.  So it suffices to prove the vanishing locus of our generators does not contain $ \Gr^\nu_\mu $ for some $ \nu \nleq \la $.

   Fix $ \nu \nleq \la $ such that $ \mu \le \nu $. Then for some $i$,
   then $d=\langle \nu - \mu, \om_{i^*} \rangle > \langle \lambda -
   \mu, \om_{i^*} \rangle$.  We will prove that there exists a point
   in $\Gr^\nu_\mu $ on which $ \Delta_{\omega_i,
     \omega_i}^{(d)} $ is non-zero.

   Let $I_+^+ = I \subset G((t^{-1}))$ denote the standard Iwahori and let
   $ I_-^+ = w_0 I_+^+ w_0^{-1} $ be the preimage of $ B_- $ in
   $G[t]$.  We claim that it suffices to prove that
   \begin{equation} \label{eq:nonempty} I_-^+ t^{w_0 \nu} I_+^+ \cap
     G_-[[t^{-1}]] t^{w_0 \mu} \ne \emptyset \text{ in } G((t^{-1}))
   \end{equation}
 To see that (\ref{eq:nonempty}) suffices, let $ g \in G_1[[t^{-1}]] $ such that $ gt^{w_0\mu} $ lies in the above intersection.  As, $I_-^+, I_+^+ \subset G[t] $, we see that $ gt^{w_0 \mu} \in \Gr^\nu_\mu $.  Finally, we can write $ g = b_- t^{w_0\nu} b_+ t^{-w_0 \mu} $ for $ b_- \in I_-^+, b_+\in I_+^+ $ and an elementary computation shows that $ \Delta^{(d)}_{\om_i, \om_i}(b_- t^{w_0\nu} b_+ t^{-w_0 \mu} ) \ne 0$.

To prove (\ref{eq:nonempty}), we work in the affine flag variety $ G((t^{-1}))/I $ and note that (\ref{eq:nonempty}) is equivalent to non-emptiness of the intersection $ I_-^+ t^{w_0 \nu} \cap G_-[[t^{-1}]] t^{w_0 \mu} $ in $ G((t^{-1}))/I $.  Let $ I_+^- $ denote the preimage of $ B $ in $G[[t^{-1}]] $ under evaluation at $ t^{-1} = 0 $.  Since $ \mu $ is dominant, $ B $ fixes $ t^{w_0 \mu} $ and thus $ G_-[[t^{-1}]] t^{w_0 \mu} = I^-_+ t^{w_0 \mu} $.  Thus we reduce to proving that
$$
I_-^+ t^{w_0 \nu} \cap I_+^- t^{w_0 \mu} \ne \emptyset \text{ in } G((t^{-1}))/I.
$$
Twisting by $ w_0 $, we reduce to proving that
$$
I_+^+ w_0 t^{w_0 \nu} \cap I_-^- w_0 t^{w_0 \mu} \ne \emptyset \text{ in } G((t^{-1}))/I.
$$
where $I_-^-$ is the preimage of $ B_- $ in $G[[t^{-1}]] $.  From
general theory of flag varieties, this is equivalent to $ w_0 t^{w_0
  \nu} \ge w_0 t^{w_0 \mu} $ in the Bruhat order on the (extended)
affine Weyl group.  This last fact is easily verified under our
hypothesis that $ \mu, \nu $ are dominant and $ \nu \ge \mu $; after
all $t^\nu\ge t^\mu$, the latter equation is arrived at by right
multiplication by $w_0$, and $t^\nu$ is a minimal double coset
representative for $W$ in the extended affine Weyl group.

\end{proof}

\section{Yangians}
\subsection{The Drinfeld Yangian}
As mentioned in the introduction, we will study subquotients of
Yangians in order to quantize our slices.  We will actually need a
slight variant on the usual Yangian, which will be produced via
a theory developed by Gavarini \cite{G1,G2}.  We begin with the usual Yangian which we call the
``Drinfeld Yangian'' to avoid confusion with the Yangian we wish to consider.

We define the {\bf Drinfeld Yangian} $ U_{\hh} \gp  $ as the associative $\C[[h]]$-algebra with generators $ e_i^{(s)}, h_i^{(s)}, f_i^{(s)} $ for $ i \in I $ and $ r,s \in \mathbb{N} $ and relations
\begin{align*}
[h_i^{(s)}, h_j^{(s)}] &= 0,  \\
[e_i^{(r)}, f_i^{(s)}] &= \delta_{ij} h_i^{(r+s)}, \\
[h_i^{(0)}, e_j^{(s)}] &=  (\alpha_i, \alpha_j) e_j^{(s)},  \\
[h_i^{(r+1)},e_j^{(s)}] - [h_i^{(r)}, e_j^{(s+1)}] &= \frac{\hh (\alpha_i, \alpha_j)}{2} (h_i^{(r)} e_j^{(s)} + e_j^{(s)} h_i^{(r)}) , \\
[h_i^{(0)}, f_j^{(s)}] &= - (\alpha_i, \alpha_j) f_j^{(s)}, \\
[h_i^{(r+1)},f_j^{(s)}] - [h_i^{(r)}, f_j^{(s+1)}] &= -\frac{\hh (\alpha_i, \alpha_j)}{2} (h_i^{(r)} f_j^{(s)} + f_j^{(s)} h_i^{(r)}) , \\
[e_i^{(r+1)}, e_j^{(s)}] - [e_i^{(r)}, e_j^{(s+1)}] &= \frac{\hh (\alpha_i, \alpha_j)}{2} (e_i^{(r)} e_j^{(s)} + e_j^{(s)} e_i^{(r)}), \\
[f_i^{(r+1)}, f_j^{(s)}] - [f_i^{(r)}, f_j^{(s+1)}] &= -\frac{\hh (\alpha_i, \alpha_j)}{2} (f_i^{(r)} f_j^{(s)} + f_j^{(s)} f_i^{(r)}), \\
i \neq j, N = 1 - a_{ij} \Rightarrow
\operatorname{sym} &[e_i^{(r_1)}, [e_i^{(r_2)}, \cdots
[e_i^{(r_N)}, e_j^{(s)}]\cdots]] = 0 \\
i \neq j, N = 1 - a_{ij} \Rightarrow
\operatorname{sym} &[f_i^{(r_1)}, [f_i^{(r_2)}, \cdots
[f_i^{(r_N)}, f_j^{(s)}]\cdots]] = 0
\end{align*}
where $\operatorname{sym}$ denotes symmetrization with respect
to $r_1,\dots,r_N$.

The following result of Drinfeld will be our starting point.

\begin{theorem}
$ U_\hh \gp $ is a quantization of $ \gp $.  More precisely, there is an isomorphism of co-Poisson Hopf algebras $U_\hh \gp / \hh U_\hh \gp  \cong U \gp$, where $ U \gp $ carries the co-Poisson structure coming from the Manin triple $ (\fg[t], t^{-1}\fg[[t^{-1}]], \fg((t^{-1}))) $.
\end{theorem}

\subsection{PBW basis for the Drinfeld Yangian}\label{PBW Yangian}

Fix any order on the nodes of the Dynkin diagram; for each positive root
$\al$, we let $\check{\al}$ denote the smallest simple root such that
$\hat{\al}=\al-\check{\al}$ is again a positive root.

We define $ e_\alpha \in \fg $ for $ \alpha \in \Delta_+ $ recursively, by
$$
  e_{\alpha_i} = e_i \text{ and } e_\alpha = [e_{\hat{\alpha}}, e_{\check{\alpha}}]
$$

We extend this definition to $ U_\hh \gp $ by defining
$$
e_{\alpha_i}^{(r)} = e_i^{(r)} \text{ and } e_{\alpha}^{(r)} = [e_{\hat{\alpha}}^{(r)}, e_{\check{\al}}^{(0)}].
$$

Similarly, we define $ f_\alpha $ and $ f_\alpha^{(r)} $.
We have the following PBW theorem for the Drinfeld Yangian:
\begin{proposition}\mbox{}
\begin{enumerate}
\item Under the isomorphism $U_\hh \gp / \hh U_\hh \gp  \cong U \gp$, $e_\alpha^{(r)} $ corresponds to $ e_\alpha t^r $.
\item Ordered monomials in the $ e_\alpha^{(r)}, h_i^{(r)}, f_\beta^{(r)} $ form a PBW basis for $ U_\hh \gp $.
\end{enumerate}
\end{proposition}

\subsection{Drinfeld-Gavarini duality}
Our goal is to give a quantization of the Poisson-Hopf algebra $ \O(\Gm) $ using the Drinfeld Yangian $ U_\hh \gp $.
For this we will use the quantum groups duality of Drinfeld-Gavarini.

We briefly describe one half of Drinfeld-Gavarini duality \cite{Dr,
  G1,G2}. Let $(H,\Delta,\ep)$ be a Hopf algebra over $\C[[\hh]]$.
Consider maps $\Delta^n:H\rightarrow H^{\otimes n}$ for $n\geq 0$
defined by $\Delta^0=\ep$, $\Delta^1 = \textrm{id}_H$, and $\Delta^n =
(\Delta\otimes \textrm{id}^{\otimes(n-2)})\circ \Delta^{n-1}$ for
$n\geq2$.  Let $\delta^n = (\textrm{id}_H - \ep)^{\otimes n}\circ
\Delta^n$, and define the Hopf subalgebra
$$
        H' = \left\{ a\in H \mid \delta^n(a) \in \hh^n H^{\otimes n} \right\}\!.
$$
In general, $H'/ \hh H'$ is a commutative Hopf algebra over $\C$ and can be given the Poisson bracket
$$
        \{a+\hh H',b+\hh H'\} = \hh^{-1} [a,b] + \hh H'.
$$

Suppose that $G$ is a Poisson affine algebraic group, namely the maximal spectrum of a Poisson commutative Hopf algebra $\O(G)$, and let $\fg, \fg^*$ be its tangent and cotangent Lie bialgebras.  Let $U_\hh = U_\hh(\fg)$ be a quantization of $U(\fg)$.

\begin{theorem}[{\cite[Theorem 2.2]{G1}}] \label{DG Duality}
There is an isomorphism of Poisson-Hopf algebras
\[ {U_\hh}'/\hh {U_\hh}' \cong \O(G^\ast) \]
where $G^\ast$ is a connected algebraic group with tangent Lie bialgebra $\fg^\ast$.
\end{theorem}

By \cite{G2}, for any basis $\left\{\overline{x}_\alpha\right\}$ of $\fg$, there exists a lift $\left\{x_\alpha\right\}$ in $U_\hh$ such that
\begin{itemize}
\item $\epsilon(x_\alpha) = 0$,
\item ${U_\hh}'$ is generated by $\left\{
    \hh x_\alpha \right\}$, and
\item ordered monomials in these
  generators span ${U_\hh}'$ over $\kb[[\hh]]$.
\end{itemize}
In particular, if $\left\{\overline{x}_i\right\}$ generates $\fg$, then $\left\{\hh x_i + \hh {U_\hh}'\right\}$ generates ${U_\hh}'/\hh {U_\hh}'$ as a Poisson algebra.

To allow for easier identification of ${U_\hh}'/\hh {U_\hh}'$ and $ \O(G^\ast)$, we can reformulate Theorem \ref{DG Duality} as follows.  Consider
$$
        \mathcal{L} = \textrm{Der} ({U_\hh}'/\hh {U_\hh}') := \left\{\left. \varphi: {U_\hh}'/\hh {U_\hh}'\rightarrow\C \right| \varphi(ab) = \varphi(a)\ep(b)+\ep(a)\varphi(b)\right\}
$$
with Lie bracket
$$
        [\varphi, \phi](a) = (\varphi\otimes\phi)(\Delta(a) - \Delta^{op}(a))
$$
and cobracket
$$
        \delta(\varphi)(a\otimes b) = \varphi(\{a,b\})
$$
This is the Lie bialgebra of the Poisson algebraic group $ Spec({U_\hh}'/\hh {U_\hh}') $.

The isomorphism described in Theorem \ref{DG Duality} can be rephrased as follows.
\begin{corollary} \label{DG Duality2}
There is an isomorphism of Lie bialgebras $\fg^\ast \cong \mathcal{L}$ defined by
\[ \overline{y} \longmapsto \Big( \hh x + \hh {U_\hh}' \longmapsto \langle \overline{y}, \overline{x}\rangle \Big) \]
for $x$ a lift of $\overline{x}\in\fg$, extended by the Leibniz rule.  This isomorphism yields a perfect Poisson--Hopf pairing $\langle\cdot,\cdot\rangle : U(\fg^\ast)\times {U_\hh}'/\hh {U_\hh}' \rightarrow \C$.
\end{corollary}

\subsection{Our Yangian}
\label{OurYangian}

We will now apply this theory to the Drinfeld Yangian $U_\hh \gp$.  We let $Y := (U_\hh \gp)'$.  We will refer to $ Y $ as the Yangian from now on.  Note that it is a subalgebra of the usual Yangian.

For $ X = E_\alpha,H_i, F_\alpha $, and $ r \ge 1 $, we define $ X^{(r)} = hx^{(r-1)}$.  From the general remarks above these elements generate $Y$ and monomials in these generators give a PBW basis for $ Y $.  We define a grading on $ Y $ where $ X^{(r)}$ has degree $ r $.

\begin{theorem}
\label{OurYPresentation}
The $ X^{(r)} $ generate $ Y $ subject to the relations
\begin{align*}
[H_i^{(s)}, H_j^{(s)}] &= 0,  \\
[E_i^{(r)}, F_j^{(s)}] &= \hh \delta_{ij} H_i^{(r+s-1)},
%\label{EF}
%\tag{\arabic{equation}}
%\addtocounter{equation}{1}
\\
[H_i^{(1)}, E_j^{(s)}] &= \hh (\alpha_i, \alpha_j) E_j^{(s)}, \\
[H_i^{(r+1)},E_j^{(s)}] - [H_i^{(r)}, E_j^{(s+1)}] &= \frac{\hh (\alpha_i, \alpha_j)}{2} (H_i^{(r)} E_j^{(s)} + E_j^{(s)} H_i^{(r)}) , \\
[H_i^{(1)}, F_j^{(s)}] &= -\hh (\alpha_i, \alpha_j) F_j^{(s)}, \\
[H_i^{(r+1)},F_j^{(s)}] - [H_i^{(r)}, F_j^{(s+1)}] &= -\frac{\hh (\alpha_i, \alpha_j)}{2} (H_i^{(r)} F_j^{(s)} + F_j^{(s)} H_i^{(r)}) , \\
[E_i^{(r+1)}, E_j^{(s)}] - [E_i^{(r)}, E_j^{(s+1)}] &= \frac{\hh (\alpha_i, \alpha_j)}{2} (E_i^{(r)} E_j^{(s)} + E_j^{(s)} E_i^{(r)}), \\
[F_i^{(r+1)}, F_j^{(s)}] - [F_i^{(r)}, F_j^{(s+1)}] &= -\frac{\hh (\alpha_i, \alpha_j)}{2} (F_i^{(r)} F_j^{(s)} + F_j^{(s)} F_i^{(r)}),\\
i \neq j, N = 1 - a_{ij} \Rightarrow
\operatorname{sym} &[E_i^{(r_1)}, [E_i^{(r_2)}, \cdots
[E_i^{(r_N)}, E_j^{(s)}]\cdots]] = 0 \\
i \neq j, N = 1 - a_{ij} \Rightarrow
\operatorname{sym} &[F_i^{(r_1)}, [F_i^{(r_2)}, \cdots
[F_i^{(r_N)}, F_j^{(s)}]\cdots]] = 0 \\
E_{\alpha_i} &= E_i \\
[E_{\hat{\alpha}}^{(r)}, E_{\check{\al}}^{(1)}] &= \hh E_{\alpha}^{(r)} \\
F_{\alpha_i} &= F_i \\
[F_{\hat{\alpha}}^{(r)}, F_{\check{\al}}^{(1)}] &= \hh F_{\alpha}^{(r)}
\end{align*}
\end{theorem}
\addtocounter{equation}{1}

We can repackage these generators relations using generating series.  Let
$$ E_i(u) = \sum_{s = 1}^\infty E_i^{(s)} u^{-s}, \quad H_i(u) = 1 + \sum_{s = 1}^\infty H_i^{(s)} u^{-s}, \quad F_i(u) = \sum_{s = 1}^\infty F_i^{(s)} u^{-s}
$$
Then the above relations can be written in series form.  For example the series version of the commutator relation between $ E_i $ and $ F_i $ is
\begin{equation}
[E_i(u), F_j(v)] = \delta_{ij} \frac{\hh}{u-v} \bigl( H_i(u) - H_i(v) \bigr), \label{EFseries}
\end{equation}

\begin{remark}
Note that the Drinfeld Yangian $ U_\hh \gp $ and our Yangian $ Y $  have natural $ \C[\hh] $-forms; moreover their $ \hh = 1 $ specializations $ U_1 \gp $  and $ Y_1 $ coincide as Hopf algebras.  The gradings on $ U_\hh \gp $ and on $ Y $ give rise to two different filtrations on $ Y_1$.  In the work of Brundan-Kleshchev \cite{BK}, these filtrations appear as the ``loop filtration'' and the ``Kazhdan filtration'', respectively.
\end{remark}

\subsection{Identification of Yangian with functions of \texorpdfstring{$ \Gm$}{
G\_1[t\textasciicircum {-1}]
}}

From the results described above, we can deduce that there is a perfect Hopf pairing between $U(\gm)$ and $Y/\hh Y$, as per Corollary \ref{DG Duality2}.
Let us denote by $Q$ the root lattice for $\fg$, let $ Q_+ $ denote the positive root cone, and let $ Q_> = Q_+ \smallsetminus \{0\} $, $ Q_< = - Q_> $.

\begin{lemma} \label{Y0 graded}
The Drinfeld Yangian $U_\hh \gp $, $Y$, and $Y/hY$, are all $Q$-graded Hopf algebras (all tensor products being graded by total degree).  The pairing between $U(\gm)$ and $Y/hY$ respects this grading.

\begin{proof}
The Hopf grading on these spaces is induced by the action of the elements $h_i^{(0)}$ (resp. $H_i^{(1)}$). In each case, coproducts preserve total degree since the coproduct is a homomorphism and the above elements are Lie algebra-like.

It is clear from the formulas of Corollary \ref{DG Duality2} that the pairing between $U(\gm)$ and $Y/hY$ respects the grading for pairings $\left\langle y, x\right
\rangle$, when $y\in \gm$, $x\in Y_0$.  The result follows for monomials $y_1\cdots y_k\in U(\gm)$ by induction on $k$.
\end{proof}
\end{lemma}

For $\alpha\in Q$, let $Y(\alpha)$ be the corresponding component of $Y/hY$ as per Lemma \ref{Y0 graded}.

\begin{proposition} \label{DG pairing}
In $Y/hY$ we have:
\[\Delta(H_i^{(r)})  = H_i^{(r)} \otimes 1 + 1 \otimes H_i^{(r)} + \sum_{s=1}^{r-1}{ H_i^{(s)}\otimes H_i^{(r-s)}} + \bigoplus_{\substack{\alpha+\beta = 0\\\alpha\in Q_<, \beta\in Q_>}} Y(\alpha)\otimes Y(\beta) \]
\[\Delta(E_i^{(r)})  = E_i^{(r)}\otimes 1 + 1\otimes E_i^{(r)} + \sum_{s=1}^{r-1}{ H_i^{(s)}\otimes E_i^{(r-s)}} + \bigoplus_{\substack{\alpha + \beta = \alpha_i\\ \alpha\in Q_<, \beta\in Q_>}} Y(\alpha)\otimes Y(\beta) \]
\[\Delta(F_i^{(r)})  = F_i^{(r)}\otimes 1 + 1\otimes F_i^{(r)} + \sum_{s=1}^{r-1}{ F_i^{(s)}\otimes H_i^{(r-s)}} + \bigoplus_{\substack{\alpha+\beta = -\alpha_i\\ \alpha\in Q_<, \beta\in Q_>}} Y(\alpha)\otimes Y(\beta)\]
\end{proposition}
\begin{proof}
To begin we recall that $\Delta(X^{(1)}) = X^{(1)}\otimes 1 + 1 \otimes X^{(1)}$ for all $x\in\fg$. Also, using the presentation of $U_\hh \gp$ with generators $x, J(x)$ for $x\in \fg$ (for which the coproduct is known), a direct calculation yields
\[ \Delta(H_i^{(2)}) = H_i^{(2)} \otimes 1 + 1\otimes H_i^{(2)} + H_i^{(1)}\otimes H_i^{(1)} - \sum_{\beta\in\Phi_+}{C_\beta (\beta,\alpha_i) F_\beta^{(1)} \otimes E_\beta^{(1)}} \]
where $(e_\beta,f_\beta) = C_\beta^{-1}$.  We prove the coproduct for $E_i^{(r)}$ by induction on $r$, using the identity
\[ E_i^{(r+1)} = \frac{1}{(\alpha_i,\alpha_i)}\left\{H_i^{(2)}, E_i^{(r)}\right\} - H_i^{(1)} E_i^{(r)} \]
The coproduct of the right side is expanded using the Poisson-Hopf algebra relations, the formula for $\Delta(H_i^{(2)})$, and the inductive hypothesis.  The above identity is the then applied again to reduce the terms in the result, and yields the form as claimed.

An analogous induction proves the case of $\Delta(F_i^{(r)})$. Finally, we take the coproduct of the identity
\[ H_i^{(r)} = \left\{E_i^{(r)},F_i^{(1)}\right\} \]
to finish the proof.
\end{proof}

Recall that the pairing between $U(\gm)$ and $Y/hY$ is determined, as per Corollary \ref{DG Duality2}, by the pairing between $\gm$ and $\gp$ given in Section \ref{sec:Poisson structure}. Choose an ``FHE'' total ordering on the generators $f_\alpha t^r, h_i t^r, e_\alpha t^r$ for $U(\gm)$.  Then it is easy to see that the previous lemma and proposition completely control the pairing between $U(\gm)$ and $Y/hY$ for the corresponding PBW basis. For example, $- F_i^{(r)}$ acts as the dual of the basis element $e_i t^{-r}$, etc.

\begin{theorem} \label{th:YangianQuantize}
There is an isomorphism of $\mathbb N $-graded Poisson Hopf algebras $\phi \colon Y/hY \cong \O(\Gm)$
such that
\begin{align*}
\phi(H_i(u)) &= \prod_j{ \Delta_{\omega_j, \omega_j}(u)^{-a_{ji}}} \\
\phi(F_i(u)) &= d_i^{-1/2}\frac{\Delta_{\omega_i,s_i \omega_i}(u)}{\Delta_{\omega_i, \omega_i}(u)} \\
\phi(E_i(u)) &= d_i^{-1/2}\frac{\Delta_{s_i \omega_i, \omega_i}(u)}{ \Delta_{\omega_i,\omega_i}(u)}
\end{align*}
where $ \O(\Gm) $ is graded using the loop rotation $\C^\times $ action.
\end{theorem}

\begin{proof}
We check explicitly that the right-hand sides act as described by the previous proposition.  Let $X =(x_1 t^{r_1})\cdots(x_k t^{r_k})\in U(\gm)$ be a basis monomial with the FHE order as chosen above.  Then we have
\[ \frac{\Delta_{\omega_i, s_i\omega_i}(u)}{\Delta_{\omega_i, \omega_i}(u)}(X) = -d_i^{-1/2}\left. \frac{\partial^k}{\partial z_1\cdots\partial z_k} \frac{\langle v_{-\omega_i}, (1+z_1 u^{r_1} x_1)\cdots (1+z_k u^{r_k} x_k) f_i v_{\omega_i}\rangle}{\langle v_{-\omega_i}, (1+z_1 u^{r_1} x_1)\cdots (1+z_k u^{r_k} x_k) v_{\omega_i}\rangle}\right|_{z_1=\ldots=z_k=0}  \]
noting that $\overline{s_i}v_{\omega_i} = f_i' v_{\omega_i} = - d_i^{-1/2} f_i v_{\omega_i}$ in the generalized minor (see Section \ref{sec:notation}).  Since we have an FHE order, to get something nonzero in the right-hand numerator $x_k$ must be a multiple of $e_i$, since $e_i f_i v_{\omega_i} = h_i v_{\omega_i} = d_i v_{\omega_i}$. In this case $z_k u^{r_k} e_i$ does not contribute to the denominator, and the remaining factors cancel leaving
\[ \frac{\Delta_{\omega_i, s_i\omega_i}(u)}{\Delta_{\omega_i, \omega_i}(u)}(X) = -d_i^{1/2} \left. \frac{\partial^k}{\partial z_1\cdots\partial z_k} z_k u^{r_k} \right|_{z_1=\ldots=z_k=0} \]
so $X$ must have been $e_i t^{r}$ to start with.  But this is precisely how $ d_i^{1/2} F_i(u)$ acts on $X$.
Similar computations hold in the two remaining cases.

To prove the equality for $H_i(u)$ one can also work in $\O(\Gm)$, and build off the known results $E_i(u)$ and $F_i(u)$, since we must have
\[ \frac{\phi(H_i(u)) - \phi(H_i(v))}{u-v} = \Big\{\phi(E_i(u)), \phi(F_i(v))\Big\} \]
We can then use formula \eqref{eq:minor-bracket} and identities for generalized minors.

The nondegeneracy of both Hopf pairings implies that $\phi$ is an injection.  It follows that $\phi$ is an isomorphism from a dimension count;
both $Y/hY$ and $\mathcal{O}(\Gm)$ have Hilbert series for the
loop grading given by
\[\prod_{i=1}^\infty \frac{1}{(1-q^{i})^{\dim\fg}}.\]
Indeed for $Y/hY$ this follows from the PBW theorem coming from $Y$, since $Y$ is a free $\C[[h]]$-algebra. On the other hand, the Hilbert series on $ \mathcal{O}(\Gm) $ is the same as the Hilbert series for $ \textrm{Sym} (\gm) $ since as $ \Gm $ is pro-unipotent, we have an isomorphism of vector spaces.

\end{proof}

\subsection{Shifted Yangians}
The Yangian has a very interesting class of subalgebras: the shifted
Yangians.  Let $\mu$ be a dominant  weight.

We will now redefine elements
\begin{equation}
\label{Fredefine}
F_\al^{(s)}=\frac{1}{h}\big[F_{\hat{\al}}^{(s-\langle \mu^*, \check{\alpha} \rangle)}, F_{\check{\al}}^{(\langle \mu^*, \check{\alpha}\rangle+1)} \big].
\end{equation}
for $ \al $ a positive non-simple root and for $ s > \langle \mu^*, \al \rangle $.
Note that these $ F_\al^{(s)} $ depend on $ \mu $.

\begin{definition}
  The {\bf shifted Yangian}  $Y_\mu $ is the subalgebra of $ Y $ generated by
  $E_\alpha^{(s)}$ for all $ \alpha, s$, $H_i^{(s)}$ for all $ i, s$, and $F_\alpha^{(s)}$ for $ s > \langle \mu^*, \alpha \rangle$.
\end{definition}

\begin{proposition}
\begin{enumerate}
 \item  Monomials in the $ E_{\alpha}^{(s)}, H_i^{(s)}, F_\alpha^{(s)} $ give a basis for $ Y_\mu $.
  \item The natural map $ Y_\mu/\hh Y_\mu \rightarrow Y / hY $ is injective.
\end{enumerate}
\end{proposition}

\begin{proof}
We first construct a PBW basis for Y slightly different from the one described in Section \ref{OurYangian}.  The generators $E_{\alpha}^{(s)}$ are defined as usual  (cf. Section \ref{OurYangian}).  The generators $F_{\alpha}^{(s)}$ are given the usual definition when $s\leq\langle\mu^*,\alpha\rangle$, but for $s>\langle\mu^*,\alpha\rangle$ we take definition (\ref{Fredefine}).  By the general remarks following Theorem \ref{DG Duality}, ordered monomials in generators $F_{\alpha}^{(s)}, H_{i}^{(s)}, E_{\alpha}^{(s)}$ are a PBW basis of $Y$.

Any element $x\in Y_\mu$ can be expressed as a linear combination of these PBW monomials.  We now show that any monomials appearing in such an expression do not contain factors of the form $F_{\alpha}^{(s)}$ for $s\leq\langle \mu^*,\alpha\rangle$.

By definition, $x$ is a linear combination of (unordered) monomials in $F_{\alpha}^{(s)}, H_{i}^{(t)}, E_{\alpha}^{(u)}$, where $s>\langle\mu^*,\alpha\rangle$.  To put $x$ in PBW form one has to commute these generators past each other.
By definition, when $s>\langle\mu^*,\alpha\rangle$, $F_{\alpha}^{(s)}$ is a linear combination of monomials built from $F_i^{(t)}$, where $t>\langle\mu^*,\alpha_i\rangle$.  Therefore it suffices to show that when commuting such $F_i^{(t)}$ past the other generators of $Y_\mu$ one never obtains  factors of the form $F_{j}^{(u)}$ for $u\leq\langle \mu^*,\alpha_j\rangle$.  This is a direct consequence of the relations appearing in Theorem \ref{OurYPresentation}.

This proves the first statement of the theorem.  The second part is a direct consequence of the first.

\end{proof}

In the limit as $ \mu \rightarrow \infty $, then we obtain $ Y_\infty
$ which is the subalgebra generated by all $ E_\alpha^{(s)}, A_i^{(s)}
$.  This is called the {\bf Borel Yangian} in \cite{FR}.

We will now show that this shifted Yangian is a quantization of $
\Gr_\mu $.   Recall that $ \O(\Gr_\mu) $ is embedded as a Poisson subalgebra of $ \O(\Gm) $.

\begin{theorem} \label{th:shiftedYangianQuantize}
The isomorphism $ \phi $ restricts to an isomorphism of Poisson algebras from $Y_\mu / \hh Y_\mu $ to $\O(\Gr_\mu)$.
\end{theorem}

\begin{proof}
First note that $ Y_\mu / \hh Y_\mu $ is generated as a Poisson algebra by all $ E_i^{(s)}$, $A_i^{(s)}$, and those $ F_i^{(s)} $ for $ i > \langle \mu^*, \alpha_i \rangle $.  We note that Lemma \ref{le:GrmuMinors} shows that the image of these generators under $\phi $ land in the subalgebra $ \O(\Gr_\mu) $.

Since $ \O(\Gr_\mu) $ is a Poisson subalgebra of $ \O(\Gm) $, we see that $ \phi $ restricts to a map $ Y_\mu/hY_\mu \rightarrow \O(\Gm) $.  This map is injective, since it is
the restriction of an injective map.  Thus, we only need to show that
it is surjective, which we do by a dimension count.

Note that by Lemma \ref{le:LieStab} the isotropy Lie algebra of $t^{w_0\mu}$ in $\Gm$
is the finite dimensional nilpotent Lie algebra $$\bigoplus_{\al \in \Delta_+} \bigoplus_{i = 1}^{-\langle
w_0\mu,\al\rangle} t^{-i}\fg_\al. $$  As a $\C^*$-module, the
functions on the group are identical to those on the Lie algebra by
the unipotence of the stabilizer.
Thus, if we let $d(k)$ be the number of roots such that $\langle
w_0\mu,\al\rangle<  k$, the Hilbert series of the functions on the
stabilizer is \[\prod_{i=1}^\infty \frac{1}{(1-q^{i})^{d(i)}}.\]

The Hilbert series of
$\O(\Gm)^{\Gm_\mu}$ is the quotient of that of $\O(\Gm)$  by that of
functions on the stabilizer.  That is, it is \[\prod_{i=1}^\infty \frac{1}{(1-q^{i})^{\dim\fg-d(i)}}.\]

On the other hand, the PBW basis for the shifted Yangian gives us the same Hilbert series for $ Y_\mu $.
\end{proof}

Thus the shifted Yangian $ Y_\mu  $ gives a quantization of $ \Gr_\mu
$.
\begin{remark}
  We should note that it is this theorem that forces us to use the
  thick Grassmannian; it will fail if we take the analogue of $
  \Gr_\mu $ in the thin affine Grassmannian, since this has ``too
  many'' functions, and will correspond to a completion of $Y_\mu$.
\end{remark}

\subsection{Deformation of the Yangian}
\label{sec:bd-yangian}

We consider a deformation of the Yangian, which we think
of as related to the Beilinson-Drinfeld Grassmannian deforming the
affine Grassmannian.  We consider for each node $i$ in the Dynkin diagram
an infinite sequence of parameters $r^{(1)}_i,r^{(2)}_i,\dots \in \C[[h]] $ and their generating series
$r_i(u)=1+r_i^{(1)}u^{-1}+\cdots $.

Now consider the algebra $Y(\mathbf r)$ generated by the
coefficients of $E_i(u), F_i(u),A_i(u)$.  The
relations are as in the previous section, with the relation \eqref{EFseries} replaced by
\begin{equation}
(u-v)[E_i(u), F_i(v)] = \hh(r_i(u)H_i(u) - r_i(v)H_i(v)), \label{EFprime}
\end{equation}
and let $Y_\mu(\mathbf r) $ be the shifted analogue of this algebra.
$Y(\mathbf r) $ is actually isomorphic to the trivial deformation of the Yangian
via the map $H_i(u)\mapsto H_i(u)/r_i(u)$.

\section{Quantization of slices}
In order to quantize the slices $ \Grlmbar $, we will need to define a quotient of $ Y_\mu $ (and its deformations $ Y_\mu(\mathbf r)$).  To do this we will use the work of Gerasimov-Kharchev-Lebedev-Oblezin \cite{GKLO}.

\subsection{Change of Cartan generators}
It will be convenient for us to change the Cartan generators of $ Y $.  Following  \cite{GKLO}, we define $ A_i^{(s)} $ by the equation
\begin{equation} \label{eq:AfromH}
H_i(u) = \frac{\prod_{j \ne i} \prod_{p = 1}^{-a_{ji}} A_j(u -\frac{\hh}{2} (\alpha_i + p \alpha_j, \alpha_j)) }{A_i(u ) A_i(u -\frac{\hh}{2}(\alpha_i, \alpha_i)) }
\end{equation}
where $ A_i(u) = 1 + \sum_{s = 1}^\infty A_i^{(s)} u^{-s} $.

\begin{example}
In the $ G = SL_2 $ case, this gives $ H(u) = \frac{1}{A(u)A(u-\hh)}$ and so for example we have
\begin{align*}
H^{(1)} = -2A^{(1)}, \quad
H^{(2)} = 3{A^{(1)}}^2 - \hh A^{(1)} - 2A^{(2)}
\end{align*}
\end{example}

\begin{proposition}[\mbox{\cite[Lemma 2.1]{GKLO}}]
Equation \ref{eq:AfromH} uniquely determines all the $ A_i^{(s)} $. \qed
\end{proposition}

One can think of the new generators $ A_i^{(s)} $ as being related to
the fundamental coweights of $ G $, whereas the $H_i^{(s)}$ match with
the simple coroots.  In particular, we have the following result which follows by setting $h = 0 $ in (\ref{eq:AfromH}).
\begin{proposition} \label{th:ImageAi}
 Let $ \phi : Y/hY \rightarrow \O(\Gm) $ be the isomorphism from Theorem \ref{th:YangianQuantize}.  Then $ \phi(A_i^{(s)}) = \Delta_{\om_i, \om_i}^{(s)} $.\qed
\end{proposition}

\subsection{The GKLO representation}
\label{GKLOrep}
In this section, we describe certain representations via difference
operators of shifted Yangians, based on work of
Gerasimov-Kharchev-Lebedev-Oblezin \cite{GKLO}.  Fix an orientation of
the Dynkin diagram; we will write $ i \leftarrow j $ to denote arrows in this quiver.  This will replace the ordering on the simple roots in \cite{GKLO}.

Fix a dominant weight $\la$ such that $\mu\leq \la$ and let $m_i = \langle \lambda - \mu, \omega_{i^*} \rangle$ and let $ \lambda_i = \langle \lambda, \alpha_{i^*} \rangle $.

Define a $\C[[\hh]]$-algebra $ D_\mu^\lambda $, with generators $z_{i,k}, \beta_{i,k}, \beta_{i,k}^{-1} $, for $ i \in I $ and $ 1 \le k \le m_i $, and $ (z_{i,k} - z_{i,l})^{-1} $, and relations that all generators commute except that $ \beta_{i,k} z_{i,k} = (z_{i,k} + d_i h) \beta_{i,k} $.

This algebra $ D_\mu^\lambda $ is an algebra of $ \hh$-difference operators.

%More precisely, it has an $ \hh= 1 $ specialization (for its natural $ \C[\hh] $ form) which has a faithful
%representation on $ \O(\prod_{i \in I} (\C^{m_i} \smallsetminus \Delta))$, where $ z_{i,k} $ acts by multiplication and $ \beta_{i,k} $ acts by shifting the %argument.  This is not quite true.

\begin{proposition}
The algebra $ D_\mu^\lambda  $ is a free $ \C[[h]]$-algebra and  we have an isomorphism of Poisson algebras
$$
D_\mu^\lambda / h D_\mu^\lambda \cong \C[z_{i,k}, (z_{i,k} - z_{i,l})^{-1} , \beta_{i,k}, \beta_{i,k}^{-1}]
$$
where the right hand side is given the Poisson structure defined by $ \{ \beta_{i,k}, z_{i,k} \} = d_i \beta_{i,k} $ and all other generators Poisson commute.
 \end{proposition}
 \begin{proof}
Obviously, we have a map  \[\C[z_{i,k}, (z_{i,k} -
z_{i,l})^{-1} , \beta_{i,k}, \beta_{i,k}^{-1}] \to
D_\mu^\lambda / h D_\mu^\lambda  \] by observing that
$D_\mu^\lambda / h D_\mu^\lambda  $ is commutative.
From the Bergman diamond lemma, we see that the algebra $ D_\mu^\lambda$ has a PBW basis consisting of
\[h^p\cdot \prod\beta_{i,k}^{\pm a_{i,k}}\cdot \prod z_{j,k}^{b_{j,k}}\cdot
\prod_{k<\ell} (z_{i,k}-z_{i,\ell})^{e_{i,k,\ell}}\] subject to
restriction that if $b_{j,k}\neq 0$, then $k$ must be maximal in its
equivalence class for the relation given by the transitive closure of the binary relation $k\sim \ell$ if $e_{j,k,\ell}\neq
0$.  Freeness over $ \C[[h]] $ follows immediately and since the same monomials give a basis of $\C[[h]][z_{i,k}, (z_{i,k} -
z_{i,l})^{-1} , \beta_{i,k}, \beta_{i,k}^{-1}]$, this confirms that we
have the desired isomorphism.  The Poisson bracket calculation follows
immediately from the relations.
 \end{proof}

Fix some complex numbers $ c_i^{(r)} $ for $ i \in I $, $ 1 \le r \le \lambda_i $.
For any variable $x$, consider the monic degree $ \lambda_i $ polynomial whose coefficients are the numbers $ c_i^{(r)} $,
$C_i(x) = x^{\lambda_i} + c_i^{(1)} x^{\lambda_i - 1} + \dots + c_i^{(\lambda_i)}.$  Note that $ x^{-\lambda_i} C_i(x) = 1 + c_i^{(1)} x^{- 1} + \dots + c_i^{(\lambda_i)}x^{-\lambda_i}$. We also introduce polynomials $Z_i(x)=\prod_{k=1}^{m_i}(x-z_{i,k})$ and  $Z_{i,k}(x)=\prod_{\ell\neq k}(x-z_{i,\ell})$.  Let $\mu_i=\left<\mu,\alpha_{i^*} \right>$ and set $F_{\mu,i}(u)=\sum_{s=1}^{\infty}F_i^{(s+\mu_i)}u^{-s}$.    Finally, for any $ \mathbf c $ as above, define $ \mathbf r $ by
\begin{equation} \label{eq:rfromc}
r_i(u) = u^{-\lambda_i} C_i(u) \frac{ \prod_{j \ne i} \prod_{p=1}^{-a_{ji}} (1- u^{-1}( \hh d_i\frac{a_{ij}}{2} + hd_jp))^{m_j}}{(1-\hh d_{i} u^{-1})^{m_i}}
\end{equation}
We are now ready to define the \textbf{GKLO representation}:
\begin{theorem}
There is a map of $ \C[[\hh]] $-algebras, $\Psi_\mu^\lambda: Y_\mu(\mathbf r) \rightarrow D_\mu^\lambda $ defined by:
  \begin{align*}
A_i(u) &\mapsto u^{-m_i}Z_i(u)\\
E_i(u) &\mapsto
d_i^{-1/2} \sum_{k=1}^{m_i}\frac{\displaystyle \prod_{j \rightarrow i}
  \prod_{p=1}^{-a_{ji}}Z_j(z_{i,k}-\hh d_i
  \frac{a_{ij}}{2}-hd_jp) }{\displaystyle  (u-z_{i,k})Z_{i,k}(z_{i,k})  }\beta^{-1}_{i,k}
  \end{align*}

And, $F_{\mu,i}(u)$ maps to
$$
-d_i^{-1/2}  \sum_{k=1}^{m_i}  C_i(z_{i,k} + hd_{i})
 \frac{\prod_{j \leftarrow i}
  \prod_{p=1}^{-a_{ji}} Z_j(z_{i,k}-\hh d_i(\frac{a_{ij}}{2}
  -1)-hd_jp)}{(u-z_{i,k} - hd_{i})Z_{i,k}(z_{i,k})  }\beta_{i,k}
 $$

\end{theorem}

\begin{proof}
When $\mu=0$ this is a reformulation of \cite[Theorem 3.1.(i)]{GKLO}.
Suppose then that $\mu\neq0$.  Then the proof of \cite[Theorem 3.1]{GKLO} applies to all the relations in $Y_\mu $ except for the commutator relation between $E_i(u)$ and $F_{\mu,i}(v)$.

In the shifted Yangian this relation takes the form
\begin{equation}
\label{RelI}
(u-v)[E_i(u),F_{\mu,i}(v)]=h(J_{\mu,i}(v)-J_{\mu,i}(u))
\end{equation}
where $J_i(v)=r_i(v)H_i(v)=\sum_{p=0}^{\infty}J_i^{(p)}v^{-p}
$
and
$$
J_{\mu,i}(v)=\sum_{p=1}^{\infty}J_i^{(p+\mu_i)}v^{-p}.
$$
To express the left hand side of (\ref{RelI}) we set:
\begin{eqnarray*}
L_i(v)&=&\frac{C_i(z_{i,k}+hd_i)\prod_{j\neq i}\prod_{p=1}^{-a_{ji}}Z_j(z_{i,k}-hd_i(\frac{a_{ij}}{2}-1)-hd_jp)}{Z_{i,k}(z_{i,k}+hd_i)Z_{i,k}(z_{i,k})(v-z_{i,k}-hd_i)}
\\
R_i(v)&=&\frac{C_i(z_{i,k})\prod_{j \neq i}\prod_{p=1}^{-a_{ji}}Z_j(z_{i,k}-hd_i\frac{a_{ij}}{2}-hd_jp))}{Z_{i,k}(z_{i,k}-hd_i)Z_{i,k}(z_{i,k})(v-z_{i,k})}
 \end{eqnarray*}
Then the left hand side of (\ref{RelI}) is equal to
$$
d_i^{-1}\sum_{k=1}^{m_i} (L_i(v)-R_i(v))-(L_i(u)-R_i(u))
$$
Note that we expressed this sum as a ``$v$-part'' minus a ``$u$-part''.

Now we consider the right hand side of (\ref{RelI}).  Note that $$\lambda_i=\mu_i+2m_i+\sum_{j \leftrightarrow i}a_{ji}m_j$$  Therefore,
$$
r_i(u) = u^{-\mu_i}\frac{C_i(u)\prod_{j\neq i}\prod_{p=1}^{-a_{ji}}(u- \hh d_i\frac{a_{ij}}{2} - hd_jp)^{m_j}}{u^{m_i}(u-hd_{i})^{m_i}}
$$
Now
$$
H_i(u) \mapsto \frac{u^{m_i}(u-hd_{i})^{m_i}}{\prod_{j\neq i}\prod_{p=1}^{-a_{ji}}(u-hd_i\frac{a_{ij}}{2}-hd_jp)^{m_j}} \frac{\prod_{j \neq i}\prod_{p=1}^{-a_{ji}}Z_j(u-hd_i\frac{a_{ij}}{2}-hd_jp)}{Z_i(u)Z_i(u-hd_i)}
$$
and hence
$$
r_i(u)H_i(u) \mapsto u^{-\mu_i}C_i(u)\frac{\prod_{j \neq i}\prod_{p=1}^{-a_{ji}}Z_j(u-hd_i\frac{a_{ij}}{2}-hd_jp)}{Z_i(u)Z_i(u-hd_i)}
$$
Therefore
$$
C_i(u)\frac{\prod_{j \neq i}\prod_{p=1}^{-a_{ji}}Z_j(u-hd_i\frac{a_{ij}}{2}-hd_jp)}{Z_i(u)Z_i(u-hd_i)}=\sum_{p=0}^{\infty}J_i^{(p)}u^{\mu_i-p}
$$
On the other hand
$$
J_{\mu,i}(u)=\sum_{p=\mu_i+1}^{\infty}J_i^{(p)}u^{\mu_i-p}
$$
showing that $J_{\mu,i}(u)$ is a truncation of $C_i(u)\frac{\prod_{j \neq i}\prod_{p=1}^{-a_{ji}}Z_j(u-hd_i\frac{a_{ij}}{2}-hd_jp)}{Z_i(u)Z_i(u-hd_i)}$.  More precisely, for $r=1,2,...$
\[
hJ_{\mu,i}(u)\Big\vert_{u^{-r}}=hC_i(u)\frac{\prod_{j \neq i}\prod_{p=1}^{-a_{ji}}Z_j(u-hd_i\frac{a_{ij}}{2}-hd_jp)}{Z_i(u)Z_i(u-hd_i)}\Big\vert_{u^{-r}}
\]
Using partial fractions we have that $\frac{h}{Z_i(u)Z_i(u-hd_{i})}$ equals
$$
\sum_{k=1}^{m_i}\frac{1}{Z_{ik}(z_{ik})Z_{ik}(z_{ik}+hd_{i})(u-z_{ik}-hd_{i})}-\frac{1}{Z_{ik}(z_{ik})Z_{ik}(z_{ik}-hd_{i})(u-z_{ik})}
$$
Therefore for $r=1,2,...$ the $u^{-r}$-coefficient of $hJ_{\mu,i}(u)$ is equal to the $u^{-r}$-coefficient of
\begin{eqnarray*}
\sum_{k=1}^{m_i}\frac{C_i(u)\prod_{j \neq i}\prod_{p=1}^{-a_{ji}}Z_j(u-hd_i\frac{a_{ij}}{2}-hd_jp)}{Z_{ik}(z_{ik})Z_{ik}(z_{ik}+hd_{i})(u-z_{ik}-hd_{i})}&-&\\\frac{C_i(u)\prod_{j \neq i}\prod_{p=1}^{-a_{ji}}Z_j(u-hd_i\frac{a_{ij}}{2}-hd_jp)}{Z_{ik}(z_{ik})Z_{ik}(z_{ik}-hd_{i})(u-z_{ik})}
\end{eqnarray*}
Now observe that for any polynomial $p(u)$ and for $r=1,2,...$
$$
\frac{p(u)}{u-z}\Big|_{u^{-r}} = \frac{p(z)}{u-z}\Big|_{u^{-r}}
$$
Therefore for $r=1,2,...$ the $u^{-r}$-coefficient of $hu^{\mu_i}J_{\mu,i}(u)$ is equal to the $u^{-r}$-coefficient of

\begin{eqnarray*}
\sum_{k=1}^{m_i}L_i(u)-R_i(u)\end{eqnarray*}
proving (\ref{RelI}).
\end{proof}

\begin{example}
\label{ExampleA1}
  If $\fg=\mathfrak{sl}_2$ and $\la=\alpha^\vee, \mu=0$, then the formulas above
  simplify considerably.  In this case,
\[
  A(u) \mapsto 1-z u^{-1}  \qquad
  E(u) \mapsto \frac{1}{u-z}\beta^{-1}
\]
and $F(u) \mapsto -((z+\hh)^2 + c^{(1)} (z+\hh) + c^{(2)})\frac{1 }{u-z-h}\beta$.
  In particular,
\begin{equation*}
H^{(1)} \mapsto 2z   \qquad E^{(1)} \mapsto \beta^{-1} \qquad F^{(1)} \mapsto -((z+h)^2 + c^{(1)} (z+h) + c^{(2)}) \beta
\end{equation*}
Restrict this representation to the copy of $ \mathfrak{sl}_2 $ generated by $E^{(1)},H^{(1)}+c^{(1)}+h, F^{(1)} $, and consider these as difference operators acting on the polynomial ring $ \C[z] $.  (More precisely, these act on $\C[[h]][z]$, but one can specialize $h$ to $1$.)  This is a standard Whittaker module for $ \mathfrak{sl}_2 $ with generic nilpotent character.
\end{example}

\begin{remark}
We can define a $\Z$-grading on $ D_\mu^\lambda $ by setting
$$
\deg \hh = 1, \ \deg z_{i,k} = 1, \ \deg \beta_{i,k} = m_i + \sum_{i \rightarrow j} a_{ij}m_j + \lambda_i - \mu_i
$$
With this definition, the GKLO representation preserves grading.
\end{remark}

\subsection{Quantization of the slices \texorpdfstring{$\Grlmbar$}{Gr}}
For any $ \mathbf c $ as above, let
 $Y_\mu^\lambda(\mathbf c) $ be the image of $ Y_\mu(\mathbf r) $ in $ D_\mu^\la $ under the GKLO representation $ \Psi_\mu^\la$ and let $I_\mu^\la(\mathbf c) $ denote the kernel of $ \Psi_\mu^\la$ (here $ \mathbf r $ is determined from $ \mathbf c $ by (\ref{eq:rfromc})).

Note that $ Y^\lambda_\mu(\mathbf c) $ is free as a $\C[[\hh]] $-algebra since it is a subalgebra of $  D_\mu^\lambda $, a free $ \C[[\hh]]$-algebra.

We have the isomorphism $ Y_\mu(\mathbf c) \rightarrow Y_\mu $ from section \ref{sec:bd-yangian} and thus we get an isomorphism of Poisson algebras $ Y_\mu(\mathbf c) / h Y_\mu(\mathbf c) \rightarrow \O(\Gr_\mu) $ from Theorem \ref{th:shiftedYangianQuantize}.  On the other hand, because $ Y^\lambda_\mu(\mathbf c) $ is free as a $\C[[\hh]]$-algebra, we get a surjection of Poisson algebras $ Y_\mu(\mathbf c)/ hY_\mu(\mathbf c) \rightarrow Y_\mu^\lambda(\mathbf c)/ h Y_\mu^\la(\mathbf c) $.

We will now establish the following theorem which shows that $ Y^\la_\mu $ is a quantization of scheme supported on $ \Grlmbar$.
\begin{theorem}\label{Ylm-quant}
There is a surjective map of Poisson algebras $ Y^\la_\mu(\mathbf c) / \hh  Y^\la_\mu(\mathbf c) \rightarrow \O(\Gr^{\bar \la}_\mu)$ which is an isomorphism modulo the nilradical of the left hand side.
\end{theorem}

\begin{remark}
Consider the map $$ Y^\la_\mu(\mathbf c) / \hh  Y^\la_\mu(\mathbf c) \rightarrow \C[z_{i,k}, (z_{i,k} - z_{i,l})^{-1} , \beta_{i,k}, \beta_{i,k}^{-1}] $$ obtained by reducing the GKLO representation mod $\hh$.  If we knew that this map was injective, then we would know that $ Y^\la_\mu(\mathbf c) / \hh  Y^\la_\mu(\mathbf c) $ was reduced and that the map from Theorem \ref{Ylm-quant} was an isomorphism.  We will in fact make a stronger conjecture.
\end{remark}

If Conjecture \ref{co:main2} holds, then we can strengthen Theorem \ref{Ylm-quant} as follows.
\begin{theorem} \label{Ylm-quant2}
If Conjecture \ref{co:main2} holds then
\begin{enumerate}
\item There is an isomorphism of Poisson algebras $Y^\la_\mu(\mathbf c) / \hh  Y^\la_\mu(\mathbf c) \rightarrow \O(\Gr^{\bar \la}_\mu)$.
\item $ Y^\la_\mu(\mathbf c) $ is the quotient of $ Y_\mu(\mathbf c)$ by the 2-sided ideal generated by $ A^{(s)}_i $ for $ s > m_i, i \in I $.
\end{enumerate}
\end{theorem}

\begin{proof}[Proof of Theorem \ref{Ylm-quant}]
Via the isomorphism $ Y_\mu(\mathbf c)/h Y_\mu(\mathbf c) \rightarrow \O(\Gr_\mu) $, we can regard $ Y_\mu^\la(\mathbf c) $ as a quotient of $ \O(\Gr_\mu) $ by an ideal, which we denote $ I_2 $.

First, note that $ \Psi_\mu^\la(A_i^{(s)}) = 0 $ for $ i \in I$, $ s >  m_i $ and thus $ \Delta_{\omega_i, \omega_i}^{(s)} \in I_2 $ for $ i \in I $ and $ s > m_i$.  Since $ I_2 $ is a Poisson ideal, we see that $ J_\mu^\la \subset I_2 $.

By Proposition \ref{pr:GenJmulam}, we see that the vanishing locus of $ J_\mu^\la $ is $ \Grlmbar $ and thus the vanishing locus of $ I_2 $ is contained in $ \Grlmbar$.  Thus it suffices to show that the vanishing locus of $ I_2 $ is not strictly contained in $ \Grlmbar $.

Since $ I_2 $ is a Poisson ideal, we see that  $ V(I_2) $ is a Poisson subvariety of $ \Grlmbar $ and thus is the union of $ \Gr^{\bar \nu}_\mu $, for $ \nu \le \lambda $.  Suppose that we have $ V(I_2) = \cup_j \Gr^{\bar \nu_j}_\mu $ for $ \nu_j < \la $.  For each $ j$, there exists $ i $ such that $ \langle \nu_j - \mu, \om_{i^*} \rangle < \langle \la - \mu, \om_{i^*} \rangle = m_i $.  Thus applying Proposition \ref{pr:SetTheoryGrlam}, $ \prod_i \Delta_{\omega_i, \omega_i}^{(m_i)} $ vanishes on $\cup_j \Gr^{\bar \nu_j}_\mu $.  Hence for some $ k $ we have $ (\prod_i \Delta_{\omega_i, \omega_i}^{(m_i)})^k \in I_2$.

On the other hand, we see that under the GKLO representation
$$\Psi_\mu^\la(A_i^{(m_i)}) = (-1)^{m_i} z_{i,1} \cdots z_{i, m_i} $$
and thus under the map
$$
\O(\Gr_\mu) \cong Y_\mu(\mathbf c)/h Y_\mu(\mathbf c) \rightarrow D^\la_\mu/hD^\la_\mu \cong \C[z_{i,k}, (z_{i,k} - z_{i,l})^{-1}, \beta_{i,k}, \beta_{i,k}^{-1}]
$$
we see that $ (\prod_i \Delta_{\omega_i, \omega_i}^{(m_i)})^k $ is mapped to a monomial in the $ z_{i,k} $.  In particular, this shows that $ (\prod_i \Delta_{\omega_i, \omega_i}^{(m_i)})^k$  does not lie in $ I_2 $, contradicting the previous paragraph.

Thus we conclude that $ V(I_2) = \Gr_\mu^\la $ as desired.
\end{proof}

\begin{proof}[Proof of Theorem \ref{Ylm-quant2}]
Let  $I_1 $ be the ideal of $ \Grlmbar $ in $ \O(\Gr_\mu) $.

Let $ K $ be the ideal in $ Y_\mu(\mathbf c) $ generated by $ A_i^{(s)} $ for $ s > m_i, i \in I $.  Then we have an inclusion $ K \subset I^\la_\mu(\mathbf c) $ and a resulting map
$$
K/hK \rightarrow I^\la_\mu(\mathbf c)/hI^\la_\mu(\mathbf c) = I_2
$$
which may not be injective.  Let $ I_3 $ denote the image of this map.  From the definitions, we see that $I_3 \subset I_2 $.  Moreover, we have that $ J^\la_\mu \subset I_3 $, since $ I_3 $ is a Poisson ideal and it contains the generators of $ I_3 $.

In the previous proof we have shown that $ I_2 \subset I_1 $. Thus we have a chain of inclusions $ J^\la_\mu \subset I_3 \subset I_2 \subset I_1 $.  On the other hand, Conjecture \ref{co:main2} shows us that $ I_1 = J^\la_\mu $.

Hence we conclude that $ I_1 = I_2 = I_3 = J^\la_\mu $.  So the first assertion holds.

For the second assertion, note that $ I_3 = I_2 $ implies that $ K/hK \rightarrow I^\la_\mu(\mathbf c)/hI^\la_\mu(\mathbf c)$ is surjective.  Let $ L = I^\la_\mu(\mathbf c)/K $.  The long exact sequence for $ \otimes_{\C[[h]]} \C $ gives
$$
K/hK \rightarrow I^\la_\mu(\mathbf c) / h I^\la_\mu(\mathbf c) \rightarrow L/hL \rightarrow 0
$$
and thus $ L/hL = 0 $.  From Nakayama's lemma, we conclude that $L = 0 $ and thus $ K = I^\la_\mu(\mathbf c) $ as desired.
\end{proof}

\subsection{Universality of the quantization} \label{se:universality}
  There is already a rich literature on the theory of deformation
  quantizations of symplectic varieties.  The most relevant work for
  us is that of Bezrukavnikov and Kaledin \cite{BK04a}, whose show the
  existence and uniqueness of deformation quantizations of symplectic resolutions.  This theory can be applied directly
  to a smooth convolution variety $ \Gr^{\overline{\vlam}}_\mu $.  Moreover, as noted by Braden, Proudfoot and the second author \cite[3.4]{BPW},  it can be extended in a very straightforward way to the
  non-smooth case $\Gr^{\overline{\vlam}}_\mu $, since we know that $ \Gr^{\overline{\vlam}}_\mu $ is a terminalization (Theorem \ref{th:terminalization}).

This shows that the variety
  $\Gr^{\bar \la}_\mu$ has a canonical family of quantizations which
  extend to a deformation
  quantization sheaf on $\Gr^{\bar \vlam}_\mu$.  The base of this
  family is the same as the base for the universal deformation of
  $\Grlmbar$ as a symplectic singularity (as constructed by
  Kaledin-Verbitsky \cite{KV02} or Namikawa \cite{NaP}).  By
  \cite[1.1]{NaWeyl}, this base $\mathbb{B}$ is an affine space modulo
  the action of a finite group.  This group can be described by
  looking at the codimension 2 strata of the product of $\Grlmbar$,
  which are $\Gr^{\la-\al_i}_\mu$, and taking the product of the Weyl
  groups attached to them by the McKay correspondence, which (using Example \ref{eg:Kleinian}) in our
  case results in the symmetric groups $S_{\la,\mu}=\prod_{i\colon
    m_i>0} S_{\la_i}$.  Here we use the fact that these strata are
  simply connected.

For the remainder of this section, let us regard the complex number $ r_i^{(s)} $ and $ c_i^{(s)} $ as variables and let $ \tilde{Y}_\mu $ be the $\C[r_i^{(s)}] $-algebra which recovers the old $ Y_\mu(\mathbf r ) $ upon specializing the variables.  Let $\tilde{Y}^\la_\mu = \tilde{Y}_\mu \otimes_{\C[r_i^{(s)}]} \C[c_i^{(s)}] / (\{A_i^{(s)} : s > m_i \}) $ (here we use a map $ \C[r_i^{(s)}] \rightarrow \C[c_i^{(s)}] $ given by (\ref{eq:rfromc})).  If Conjecture \ref{co:main2} (and hence Theorem \ref{Ylm-quant2}) holds, then $ \tilde{Y}^\la_\mu $ can be specialized (via a map $ \C[c_i^{(s)}] \rightarrow \C $) to each of the $ Y^\la_\mu(\mathbf c) $.  We conjecture that $ \tilde{Y}^\la_\mu $ is related to the above universal quantization as follows.

  First note that the BD analogue $\Gr_{\mu;\mathbb{A}^{\rho(\la)}}^{\bar\vlam}$ is a
  symplectic deformation of $\Grlmbar$ over the base
  $\mathbb{A}^{\rho(\la)}$, and thus is the pull-back of the universal
  deformation by a map $b\colon \mathbb{A}^{\rho(\la)}\to \mathbb{B}$.

\begin{conjecture}
\begin{enumerate}
\item The map $ b : \mathbb{A}^{\rho{\la}} \rightarrow \mathbb{B} $ descends to a surjective map $ \tilde b : \mathbb{A}^{\rho{\la}}/S_{\la, \mu} \rightarrow \mathbb{B} $.
\item The algebra $\tilde{Y}^\la_\mu$ is the base change along $\tilde b $ of the universal, Bez\-ru\-kav\-ni\-kov-Kaledin-type quantization.
\end{enumerate}
\end{conjecture}

\begin{example}
We continue Example \ref{ExampleA1}, so $ G = SL_2 $ and $ \lambda = \alpha^\vee, \mu = 0 $.  Note that in $ Y^\lambda_\mu $, we have that  $E^{(s)}=(-A^{(1)})^{s-1}E^{(1)}$,
and $F^{(s)}=F^{(1)}(-A^{(1)})^{s-1},$  and so $Y^\lambda_\mu$ is generated by $ E^{(1)} $ and $ F^{(1)} $.

Let $ U_\hh \mathfrak{sl}_2 $ denote the $ \hh$-version of the universal enveloping algebra of $ \mathfrak{sl}_2 $.  Let $ C = EF + FE + \frac{1}{2}H^2 $ be its Casimir element.  For any complex number $ c $, let $ Z_c $ denote the ideal in $U_\hh \mathfrak{sl}_2$ generated by the central element $C - c $.  Standard results give  that $ U_\hh \mathfrak{sl}_2 / Z_c $ is a quantization of the nilpotent cone of $ \mathfrak{sl}_2 $, which is isomorphic as a Poisson variety to $ \Grlmbar $.

The map $$ E^{(1)} \mapsto E,  \quad H^{(1)} \mapsto H +c^{(1)} + h, \quad F^{(1)} \mapsto F$$
defines an isomorphism $ Y^\lambda_\mu \cong U_h\mathfrak{sl}_2/Z_{c} $, where $c=2c^{(2)}-\frac{1}{2}(c^{(1)})^2+\frac{1}{2}h^2$.  If we don't specialize, then the same formulas combined with the assignment
$$
c^{(2)}\mapsto -\frac{1}{2}C+\frac{1}{4}(c^{(1)})^2-\frac{1}{4}h^2
$$
give an  isomorphism
  \[\tilde{Y}^2_0 \cong U_\hh(\mathfrak{sl}_2)[c^{(1)}].\]

In this example,
$U_\hh(\mathfrak{sl}_2)$ is the universal quantization, and $c^{(1)}$ a
trivial deformation parameter. The universal family
is \[\mathfrak{sl}_2\overset{\operatorname{tr}(a^2)}\longrightarrow\C. \]
Since the fiber of the BD analogue over $(x,y)\in \mathbb{A}^2$ can be
identified with matrices with eigenvalues $x$ and $y$, the map $b$ is
just $b(x,y)=\nicefrac{1}{4}(x-y)^2$.  Thus, choosing $x+y$ and
$(x-y)^2$ as generators of symmetric functions, $\tilde b$ is just the projection map $\mathbb{A}^2\to \mathbb{A}^1$.
\end{example}

The sum of the $c_i^{(1)}$
is always a trivial deformation parameter; usually this is the only such parameter,
but there are degenerate cases where other parameters can be
trivialized as well (for example, if $\la=\mu$).

\subsection{Quantization of Zastava spaces}
In this section, we assume that Conjecture \ref{co:main2} holds and thus we will assume the conclusions of Theorem \ref{Ylm-quant2}.

Let us fix $ \nu $ in the positive coroot cone.  Choose some $ \mu_0 $ such that $ \mu_0 + \nu $ is dominant.  Let $ \mathbf c $ be a collection of complex numbers as above and consider $ Y_{\mu_0}^{\mu_0 + \nu}(\mathbf c) $.

Now for any dominant $ \mu$ with $ \mu \ge \mu_0 $, we extend $\mathbf c $ by 0 and (slightly abusing notation) consider $ Y_{\mu}^{\mu+\nu}(\mathbf c) $.  Since the generators of $ Y_{\mu}^{\mu+\nu}(\mathbf c)$ are a subset of the generators of $ Y_{\mu_0}^{\mu_0 + \nu}(\mathbf c) $ and the relations are the same, we obtain a map $ Y_{\mu}^{\mu+\nu}(\mathbf c)\rightarrow Y_{\mu_0}^{\mu_0 + \nu}(\mathbf c) $.  It is easy to see that this map is an isomorphism on the $ N$th filtered piece if $ \langle \mu, \alpha_i \rangle \ge N $ for all $i $.

Thus this system stabilizes to the algebra $ Y^{\infty + \nu}_\infty
$, which is the quotient of the Borel Yangian $ Y_\infty $ by the
2-sided ideal generated by $ A_i^{(s)} $ for $ s > \langle \nu,
\alpha_i \rangle $; perhaps surprisingly, this limit doesn't depend on $ \mathbf c $ or our starting $ \mu_0 $.

Combining together Theorem \ref{Ylm-quant2} with Theorem \ref{th:maptoZastava}, we obtain the following (dependent on Conjecture \ref{co:main2}), which was conjectured in \cite{FR} for $ G = SL_n $ (and proven for $ G =SL_2 $).
\begin{theorem}
$Y_\infty^{\infty + \nu}/ \hh Y_\infty^{\infty + \nu} $ is isomorphic
to the Poisson algebra $ \O(Z_\nu) $. \qed
\end{theorem}

\begin{remark}
As mentioned above, the GKLO representation gives rise to a map of graded Poisson algebras
$$
Y^\lambda_\mu(\mathbf c) / \hh Y^\lambda_\mu(\mathbf c) \rightarrow D^\lambda_\mu(\mathbf c) / \hh D^\lambda_\mu(\mathbf c)
$$
(which we expect is an inclusion) and thus to a $\C^\times $-equivariant map of Poisson varieties
$$
\prod_i (\C^{m_i} \smallsetminus \Delta) \times (\C^\times)^{m_i} \rightarrow \Gr^\lambda_\mu
$$
which we expect to be \'etale.

If we then compose with the map $ Gr^\lambda_\mu \rightarrow Z_{\lambda-\mu} $, we obtain $ \prod_i (\C^{m_i} \smallsetminus \Delta) \rightarrow Z_{\lambda-\mu} $, which was studied in \cite{GKLO}.
\end{remark}

\end{document}